\documentclass[12pt,amstex,leqno]{amsart}

\usepackage{latexsym}
\usepackage{amsthm}
\usepackage{amsmath}

\usepackage{amsfonts}
\usepackage{amssymb}
\usepackage{appendix}
\usepackage[mathcal]{euscript}
\usepackage{bm}

\newcommand\op{\operatorname{op}}

\newcommand\Set{\operatorname{\bf Set}}

\newcommand{\id}{\operatorname{id}}
\newcommand{\Id}{\operatorname{Id}}

\newcommand\colim{\operatorname{colim}}

\newcommand{\Met}{\operatorname{\bf Met}}
\newcommand{\CMet}{\operatorname{\bf CMet}}
\newcommand{\PMet}{\operatorname{\bf PMet}}

\newcommand{\Lan}{\operatorname{\bf Lan}}

\newcommand{\Mnd}{\operatorname{\bf Mnd}}  
\newcommand\ck{\mathcal {K}}
\newcommand\ca{\mathcal {A}}
\newcommand\cb{\mathcal {B}}

\newcommand\cd{\mathcal {D}} %{\mathscr{D}}

\newcommand\ce{\mathcal {E}}
\newcommand\cv{\mathcal {V}}
\newcommand\cw{\mathcal {W}}

\newcommand\ch{\mathcal {H}}

\newcommand\N{\mathbb{N}}

\newcommand\var{\varepsilon}

\input xy
\xyoption{all}

\swapnumbers

\newtheorem{theorem}{Theorem}[section]
\newtheorem{lemma}[theorem]{Lemma}

\newtheorem{birk}[theorem]{Quantitative Birkhoff Variety Theorem}
\newtheorem{open}[theorem]{Open problem}

\newtheorem{prop}[theorem]{Proposition}
\newtheorem{coro}[theorem]{Corollary}

\theoremstyle{definition}
\newtheorem{defi}[theorem]{Definition}
\newtheorem*{defini}{Definition}
\newtheorem{assump}[theorem]{Assumption}
\newtheorem{example}[theorem]{Example}
\newtheorem*{exam}{Example}
\newtheorem{exs}[theorem]{Examples}
\newtheorem{remark}[theorem]{Remark}

\newtheorem{constr}[theorem]{Construction}

\newtheorem{nota}[theorem]{Notation}

\numberwithin{equation}{section}

\begin{document}

%\begin{document}
	%zkraceno zahlavi
	\title[Strongly finitary monads are too strong]{Strongly finitary metric monads are too strong}
	
	\author[J. Ad{\'{a}}mek]{Ji{\v{r}}{\'{\i}} Ad{\'{a}}mek%$^{^{\ast)}}$
}
	\dedicatory{\rm Czech Technical University in Prague\\
		Technical University in Braunschweig}
		
		\begin{abstract}
		Varieties of quantitative algebras are fully described by their  free-algebra monads on the category $\Met$ of metric spaces. For a longer time it has been  an open  problem whether the resulting enriched monads are precisely the strongly finitary ones (determined by their values on finite discrete spaces). We present  a counter-example: the variety of algebras on two $\var$-close  binary operations  yields a monad which is not strongly finitary. A full characterization of free-algebra monads  of varieties is presented: they are the \emph{1-basic} monads, i.e., weighted colimits of strongly finitary monads (in the category of finitary  monads).

 We deduce that strongly finitary endofunctors on $\Met$ are not closed under composition.
\end{abstract}

\maketitle
\section{Introduction}
\emph{Quantitative algebras}, which are algebras acting on metric spaces 
with nonexpansive operations, were introduced by Mardare, Panangaden and Plotkin \cite{MPP16}, \cite{MPP17} as a foundation  for  semantics of probabilistic or stochastic systems.  A basic tool for presenting classes of quantitative algebras are $c$-basic quantitative equations for a cardinal number $c$. We concentrate on the case $c=1$: these equations have  the form $t=_{\var} t'$ for terms $t$, $t'$ and a real number $\var\geq 0$.
An algebra satisfies that equation iff every computation of the terms $t$ and $t'$ yields results of distance at most $\var$.
 A \emph{variety} of quantitative algebras is a class presented by a set of $1$-basic quantitative equations. A prominent example in loc. cit. is the variety of quantitative semilattices.

Every variety $\cv$ has free algebras on all metric spaces. It thus yields a monad $T_{\cv}$ on the category $\Met$ of metric spaces. Moreover, $\cv$ is  isomorphic to the corresponding category $\Met^{T_{\cv}}$ of Eilenberg-Moore algebras. (Example: for  quantitative semilattices $T_{\cv} X$ is the finite power-set endowed with the Hausdorff metric.) It has been an open problem for some time to characterize monads of the form $T_{\cv}$, see e.g. \cite{MG}, \cite{LP},  \cite{R}  or \cite{ADV}.
All  strongly finitary monads of Kelly and Lack (\ref{D:quant} below) are of that form (Theorem \ref{T:adv}). One could expect that, conversely, every free-algebra monad of a variety is strongly finitary.
Indeed,  the analogous statement
 is true for a number of other  basic categories, e.g.\ sets, ultrametric spaces \cite{ADV2} or posets \cite{KV}. But $\Met$  is an exception: in Section~7 we present a simple variety $\cv$ such that $T_\cv$ is not strongly finitary: it consists of two binary operations which are $\var$-close.

%We form the category $\Mnd _f(\Met)$ of enriched finitary monads on $\Met$, %and call a \emph{semi-strongly finitary} if it is a weighted colimit of %strongly finitary monads. Our main result is the following.

\begin{defini} 
A monad on $\Met$ is \emph{1-basic} if it is a weighted colimit of strongly finitary monads  in $\Mnd _f(\Met)$, the category of enriched finitary monads on $\Met$.
\end{defini}

Our main result in Section 5 is the following

\begin{theorem} Monads of the form $T_{\cv}$ are precisely the 1-basic ones.
\end{theorem}

\begin{coro} Strongly finitary endofunctors in $\Met$ are not closed under composition.

\end{coro}
Indeed, Bourke and Garner introduced in \cite{BG} saturated classes of arities and proved in Theorem  43 a result that implies that, in case strongly finitary functors compose (in their terminology: all finite discrete spaces form  a saturated class of arities in $\Met$) the free-algebra monads of varieties of quantitative algebras  are precisely the strongly finitary ones.

In the work of Mardare et al also \emph{complete} quantitative algebras play an important role. These are the quantitative algebras in the category $\CMet$ of complete metric spaces. For example, the variety of semilattice in this setting yields the
Hausdorff monad on $\CMet$ (of all compact subsets). An analogous result holds: varieties of complete quantitative algebras correspond precisely to  the weighted colimits of strongly finitary  monads.  

A number  of important cases actually yield  strongly finitary monads, e.g. the Hausdorff monad.
Moreover, all monads $T_\cv$ of varieties $\cv$ of unary algebras 
%$\CMet$. Thus it is presented  by ordinary equation (where $\var=0$). We prove that all monad %$T_\cv $ where $\cv $ is presented by ordinary equations, or where all operations are unary,
 are strongly finitary (Section 6).

\vskip 2 mm
\noindent \textbf{Related Work.} We have announced our example of a variety $\cv$ not having a strongly finitary free-algebra monad in July 2025 at the conference  CT25 \cite{Ada}. Independently, a similar example has been annouced by Mardare et al.\ \cite{MGR}, Example 8.3, in September 2025 at the conference GandALF 2025.

More general quantitative equations were used in the pioneering paper of Mardare, Panangaden and Plotkin  \cite{MPP16}:
the \emph {$c$-basic} equations for a cardinal number $c$.  They have the form $t=_\var t'$, where $t$ and $t'$ are elements of the free algebra $T_\Sigma V$ on a space $V$ (of variables) of power less than $c$. The corresponding monads
are called \emph{$c$-basic monads} in \cite{A}. The $\aleph_0$-basic monads restricted to ultrametric spaces were characterized as precisely the monads preserving surjective morphisms and directed colimits of split monomorphisms. It is an open problem whether this characterization also holds for $\aleph_0$-basic monads on $\Met$ or $\CMet$.

The unpublished paper \cite{ADV} with the extended abstract \cite{ADV2} contains some incomplete arguments. 
The co-authors unfortunately do not intend publishing a revised version. This leads us to providing new (and, as it happens, simpler) proofs in Sections 4 and 5 below. We also repeat some of the introductory material of \cite{ADV} in Section 2.

An alternative approach to varieties of quantitative algebras is presented by J. Rosick\'{y} \cite{R}, % \cite{R2}, 
who uses algebraic theories. 
\vspace{5 mm}

\textbf{Acknowledgements.}
The author is grateful to M. Dost\'{a}l, J. Rosic\-k\'{y}, H. Urbat, and J. Velebil for fruitful  discussions that helped to improve
the presentation of our paper.

\section{Finitary and Strongly Finitary Functors}\label{sec2}

We work  here with  the monoidal closed categories $\Met$ of metric spaces. Finitary and strongly finitary endofunctors, substantially used in subsequent sections, are discussed. 

\begin{nota} %2.1
$\Met$ denotes the category  of (extended) metric spaces. Objects are metric spaces extended in the sense that the distance $\infty$ is allowed. Morphisms are nonexpansive functions $f\colon X\to Y$: for all $x$, $x'\in X$ we have $d(x, x') \geq d\big(f(x), f(x')\big)$.

$\CMet$ is the full subcategory of complete spaces: every Cauchy sequence converges.

%The underlying set of a space $X$ is denoted by $|X|$.
 
%Every set is considered to be the \emph{discrete} space: all non-zero distances are $\infty$.
 
\end{nota}
 
\begin{remark}\label{R:closed}  %2.2
(1) $\Met$ is a symmetric monoidal closed category, where the tensor
  product
$$
X\otimes Y
$$
is the cartesian product with the  \emph{sum metric}:
$$
d\big((x,y), (x', y')\big) = d(x,x') + d(y,y')\,.
$$
In contrast, the categorical product $X\times Y$ is the cartesian product with the \emph{maximum metric}: the maximum of $d(x,x')$ and $d(y, y')$. By the \emph{power} $X^n$ of a space $X$ $(n \in \Bbb{N}$) we always mean the categorical product (with the maximum metric).

 The monoidal unit $I$ is a singleton space. The hom-space
$$
[X,Y]
$$
is the space of all morphisms $f\colon X\to Y$ with the \emph{supremum metric}
$$
d(f, f') =\sup_{x\in X} d\big(f(x), f'(x)\big) \quad \mbox{for} \quad f, f'\colon X\to Y\,.
$$

(2) A $\Met$-\emph{enriched} (or just enriched) \emph{category} is a category with a metric on every hom-set making composition nonexpanding (with respect to the addition metric). A ($\Met$)-\emph{enriched functor} $F$ between  enriched categories is a functor which is locally nonexpanding: for all parallel pairs $f$, $g$ in the domain category  we have  $d(Ff, Fg) \leq d(f,g)$. Enriched natural transformations are the ordinary ones (among enriched functors).
Thus, enriched monads are those with the enriched underlying endofunctor. This follows from $\Met(I, -)$
being naturally isomorphic to $\Id$.
 
% (3) All of the above has an obvious analogy for complete metric spaces: if $X$ and $Y$ are complete,  then so are $X\otimes Y$ and $[X,Y]$. $\CMet$-enriched categories  and functors  are defined as above.
% 
% When speaking about \emph{enriched categories} we always mean enriched over either $\Met$ or $\CMet$. Analogously for enriched functors. Where necessary, 
% we distinguish  the two cases explicitly. But usually (except concrete examples) the arguments for $\Met$ and $\CMet$ are the same.

(3) Given  enriched categories $\ca$ and $\cb$, the category $[\ca, \cb]$ of all enriched functors $F\colon \ca \to \cb$ and natural transformations is enriched: the distance of natural transformations $\varphi, \psi \colon F\to F'$ is $d\varphi, \psi ) = \sup\limits_{ X\in \ca} d(\varphi_X, \psi_X)$.
 %%%%%%%%%%%%%%%%%%%%%
 \end{remark}
 
 \begin{exam}
 	(1) The category $\Mnd(\Met)$ of enriched monads on $\Met$ is enriched: the distance of parallel monad morphisms is the supremum of the distances of their components. 

 	(2) Every full subcategory of an enriched category is ennriched in the expected sense. We will particularly often work with the full subcategory
 	$$\CMet$$ 
 of complete spaces  in $\Met$.
 \end{exam}
 
 \begin{nota}\label{N:under}
(1) For every metric space $X$  we denote by $|X|$ its underlying set.

Conversely, every set is considered as the \emph{discrete space}: all non-zero distances are $\infty$. (This space  is complete.)   For every space $X$ we thus have a morphism
$$
i_X \colon |X| \to X \quad \mbox{in} \ \Met 
$$
carried by the identity.

(2) Every natural number $n$ is considered to be the set $\{0, \dots, n-1\}$.

\end{nota}

\begin{remark} \label{R:Cauchy}
	(1) The subcategory $\CMet$ is reflective in $\Met$. The reflection of a space $X$ is its \emph{Cauchy completion} $X^*$: a  complete space containing $X$ as a dense subspace. Indeed, given a nonexpanding map $f$ from $X$ to a complete space $Y$, its unique continuous extension $f^* \colon X^* \to Y$ is easily seen to be nonexpanding.

(2) Epimorphisms $e \colon X \to Y$ in $\Met$ and $\CMet$ are precisely the \emph{dense}
morphisms: those whose image is a dense subspace of $Y$.

(3) Directed colimits (indexed by directed posets) 
in the categories $\Met$ and $\CMet$ are not set-based.  Consider for example the $\omega$-chain of spaces $M_n =\{a, b\}$  with $d(a,b)=2^{-n}$ (and   $\id$ as connecting maps). Then $\colim M_n$ is a singleton space.
 A concrete description of directed colimits can be found in the Appendix.

\end {remark}

\begin{defi}\label{D:fin}
An endofunctor is \emph{finitary} if it preserves  directed colimits.
A monad is finitary if its underlying functor is.
\end{defi}

\begin{example}\label{E:pol}
(1) The endofunctor $(-)^n$ of the $n$-th categorical power is finitary on $\Met$ as well as $\CMet$ (for every $n\in \N$). This follows easily from Propositions %\ref{P:dir}.
A1 and A2 in the Appendix.

(2) %Coproducts in $\Met$ and $\CMet$ are disjoint unions with distance $\infty$ between elements of distinct summands.
 A coproduct of finitary functors is finitary: coproducts commute with colimits.

(3) For every (complete) metric space $M$ the endofunctor  $M\times $- on $\Met$ or $\CMet$ is finitary.

(4) The \emph{Hausdorff endofunctor} $\mathcal{H} \colon \CMet \to \CMet$ is  finitary \cite{AMMU}, Example 3.13. It assigns to a space $X$ the space $\mathcal{H}X$ of all compact subsets with the Hausdorff metric
$$
d_{\mathcal{H}} (A,B) =\max\big\{ \sup_{a\in A} d(a, B), \sup_{b\in B} d(b, A)\big\}\,,
$$
where the distance  $d(a,B)$ is (as usual) $\inf_{b\in B} d(a,b)$.
(In particular,  $d_{\mathcal{H}}(A, \emptyset)= \infty$  if $A\ne \emptyset$.) On morphisms $f\colon X\to Y$ the Hausdorff functor  is given by $A\mapsto f[A]$.
%The finitarity of $\ch$ is proved in \cite{AMMU}, Example 3.13.
\end{example}

 Let $\ca$ be an  enriched category with a full enriched subcategory $K \colon \ca_0 \hookrightarrow \ca$. We recall the concept of (enriched) Kan extension. Suppose that the restriction functor
$$
K \cdot (-) \colon  [\ca, \ca] \to [\ca_0, \ca]
$$
has an enriched  left adjoint. Then this adjoint  is denoted by
$$
F \mapsto \Lan_K F\colon \ca \to \ca \qquad (\mbox{for} \ F\colon \ca_0 \to \ca),
$$
and is called the \emph{left Kan extension} along $K$. Thus $T=\Lan_K F$ is an enriched endofunctor equipped with a natural transformation $\tau \colon F \to TK$ having the obvious universal property.
%\end{remark}

\begin{nota} %%2.6
The full embedding of the category of finite sets (or   finite discrete spaces) is denoted by
$$
K\colon \Set_f \to \Met \quad \mbox{or}\quad K\colon \Set_f\hookrightarrow \CMet\,.
$$
\end{nota}

The following definition is analogous  to the definition presented by Kelly and Lack \cite{KL} for locally finitely presentable categories.

\begin{defi}\label{D:quant} %%% 2.6
An enriched endofunctor $T$ on $\Met$ or $\CMet$ is \emph{strong\-ly finitary} if it is obtained from its restriction $TK$ to finite discrete spaces via the left Kan extension:
$$
T =\Lan_K TK\,.
$$
 A monad is strongly finitary if its underlying endofunctor is.
\end{defi}

\begin{remark} \label{R:adv}
	Every strongly finitary monad is finitary (see the next theorem), but not conversely (Proposition \ref{P:contra}).

We now recall a condition characterizing strong finitarity, proved in \cite{ADV}.
%(Propositions 2.20 and 2.22 ). %for $\U\Met$, also works for $\Met$ and  $\CMet$. The proof is the same.
We first need  the following
\end{remark}

\begin{nota}\label{N:comp}
For every metric space $X$ and every $\var >0$ we denote by $\Delta_\var  X \subseteq |X|^2$ the
$\var$-neighbourhood of the diagonal: the
 set of all pairs of distance at most $\var$. The left and right projections to $|X|$ are denoted by
$$
l_\var, r_\var \colon \Delta_\var X \to |X|\,,
$$
respectively. Recall the identity-carried morphism $i_X$  (Notation \ref{N:under}).
\end{nota}

\begin{theorem}[\cite{ADV}, Theorem 3.6 and Proposition 2.23]\label{P:comp}
An enriched endofunctor $T$ on $\Met$ or $\CMet$ is strongly finitary iff

{\rm a.} $T$ is finitary.

{\rm b.} For every space $X$ the map $Ti_X$ is an epimorphism.

{\rm c.} Given spaces $X$ and $Y$, each 
 nonexpanding map $f\colon T|X| \to Y$ satisfying the following condition
\begin{equation}\label{cond}
d(f\cdot Tl_\var, f\cdot Tr_\var) \leq \var \quad \mbox{for all}\quad \var > 0
\end{equation}
has a nonexpanding factorization $f'$ through $Ti_X$:
$$
\xymatrix  @C=3pc@R=3pc{
&& TX \ar[d]^{f'}\\
T\Delta_\var X  \ar@<0.5ex>[r]^{Tr_{\var}}
      \ar@<-0.5ex>[r]_{Tl_{\var}}& T|X| \ar[ur]^{Ti_X} \ar[r]_f & Y
      }
      $$
\end{theorem}

\begin{exs}\label{E:var} %%%%{E:comp}

(1) The finitary endofunctor $TX = X^n$ is strongly finitary on $\Met$ or $\CMet$ (for every $n\in \N$).
Indeed, $Ti_X = \id$, and a function $f\colon |X^n|\to Y$ satisfying \eqref{cond} is clearly nonexpanding, $f\colon X^n\to Y$.

In contrast, if $M=\{0,1\}$ is the space  with $d(0,1)$ finite, then the functor $T=\Met(M, -)$, assigning to $X$ the subspace of $X\times X$ on $\Delta_\var X$, is not strongly finitary:  the morphism $Ti_M$ is not epic.

(2) A coproduct of strongly finitary functors is strongly finitary. Coproducts in $\Met$ or $\CMet$ are disjoint unions with distance  $\infty$ between elements of distinct summands. It is easy to verify that if all summands fulfil the conditions of Proposition \ref{P:comp}, then so does the coproduct.

(3) The {Hausdorff endofunctor} $\mathcal{H} \colon \CMet \to \CMet$ is strongly finitary (\cite {ADV}, Example 2.25).
%%%%

(4) For every metric space $M$ the endofunctor $T=M\times -$ is strongly finitary. Given a nonexpanding map $f\colon M\times |X|\to Y$ satisfying \eqref{cond}, then $f\colon M\times X\to Y$ is also nonexpanding. This means nonexpanding in each component, and $f(m, -)$ is clearly nonexpanding.
% due to $\{m\}\times X_\var \subseteq TX_\var$.
 
\end{exs}

Let us recall the concept of a \emph{weighted limit} in an enriched category $\ck$. Given a diagram in $\ck$, an enriched functor $D\colon \cd\to \ck$, and an enriched weight
$$
W\colon \cd\to \Met,
$$
the weighted limit 
$$
C=\lim_W D
$$
is an object of $\ck$ together with isomorphisms
$$
\psi_X \colon \ck (X,C) \to [\cd, \Met] (W, [X, D-])
$$
natural in $X\in \ck^{\op}$.

Dually: a \emph{weighted colimit} of a diagram $D\colon \cd \to \ck$ with a weight $W\colon \cd^{\op} \to \Met$.

\begin{example}\label{E:eq}
	(1) Let $f, f' \colon X \to Y$ be a pair of morphisms in an enriched category.
	The \emph{$\var$-equalizer} of the pair is a morphism $e\colon E\to X$ universal with respect to
	$$
	d(f\cdot e, f'\cdot e) \leq \var\,.
	$$
	That is, given a morphism $a\colon A\to X$ with $d(f\cdot a, f'\cdot a) \leq \var $, then it factorizes through $e$.  Moreover, for every pair $u_1, u_2\colon U\to E$, we have $d(u_1, u_2)= d(e\cdot u_1, d\cdot u_2)$.
	
	This is $\lim\limits_W D$, where the domain of $D$ is a parallel pair of distance $\infty$, and $D$ assigns to it the pair $f, f'$, whereas $W$ is given as follows 
	$$
	\boxed{\rightrightarrows} \longmapsto \boxed{\{0\} \rightrightarrows \{\square, \Diamond\}}\quad \mbox{with}\quad d(\square, \Diamond)=\var\,.
	$$
	
	(2)Dually, the \emph{$\var$-coequalizer} is the universal morphism $c\colon Y\to C$ with respect to $d(c\cdot f, c\cdot f')\leq \var$. Here the weight is as follows
	$$
	\boxed{\rightrightarrows} \longmapsto \boxed{\{\square, \Diamond\} \leftleftarrows \{0\}}
	$$
	
	(3) The \emph{tensor} $M\otimes X$ of a space $M$ and an object $X$ of an enriched category $\ck$ is the colimit of the diagram $D\colon 1\to \ck$ representing $X$,  with $W\colon 1\to \Met$ representing $M$. This is an object $M\otimes X = \colim\limits_W D$ together with metric isomorphisms
	$$
	\xymatrix@R=.21pc{
		&M\otimes X \ar[r] &Y&\\
		\ar@{-}[rrr]&&&\\
		&M \ar[r]& \ck(X,Y)&
	}
	$$
	natural in $Y\in \ck$. This generalizes copowers: if $M$ is discrete, then $M\otimes X = \coprod\limits_M X$.
	
	(4) Dually a \emph{cotensor} $M \pitchfork X$ is $\lim\limits_W D$: an object  $M \pitchfork X$ together with an isomorphism
	$$
	\xymatrix@R=.21pc{
		&Y \ar[r] &M \pitchfork X&\\
		\ar@{-}[rrr]&&&\\
		&M \ar[r]& \ck(Y,X)&
	}
	$$
	natural in $Y \in \ck^{\op}$. If $M$ is discrete, then $M\pitchfork X = X^M$.
\end{example}

\begin{remark}\label{R:X}
	By \cite{ADV}, Proposition 2.20, every space $X$ in $\Met$ is a weighted colimit of a diagram of discrete spaces, called the \emph{foliation} of $X$.
	The domain $\cb$ of that diagram is the linearly ordered real interval $(0, \infty)$, enlarged by two cocones with domain $0$, denoted by $l_\var$, $r_\var \colon \var \to 0$ ($\var>0$). The foliation $D_X \colon \cb \to \Met$ takes $0$  to $|X|$ and $\var>0$ to $\Delta_\var X$.  It takes  the morphisms $l_\var$ and $r_\var$ to the left and right projections of $\Delta_\var$ The weight $B\colon \cb^{\op} \to \Met$ takes $0$ to $\{p\}$, and $\var$ to $\{l, r\}$ where $d(l,r) =\var$. Futher, $Bl_\var$, $Br_\var \colon \{p\} \to \{l, r\}$ are as expected. We have $X=\colim\limits_B D_X$.
\end{remark}

	\begin{prop}\label{P:X}(\cite{ADV}, Theorem 3.6)
		Every strongly finitary functor preserves colimits of foliations.
\end{prop}

\section{Varieties of Quantitative Algebras}

We recall varieties of quantitative algebra from \cite{MPP16}. 
Every variety $\cv$ is shown to be isomorphic to the category $\Met^{T_{\cv}}$ of Eilenberg-Moore algebras, where $T_{\cv}$ is the free-algebra monad.
We prove that $T_\cv$ is enriched and finitary, and present some examples.

Throughout our paper  $\Sigma=(\Sigma_n)_{n\in \N}$ denotes a finitary  signature: $\Sigma_n$ is the set of $n$-ary operation symbols.   

\begin{defi}[\cite{MPP16}]\label{D:q}
A \emph{quantitative algebra} is a metric space $A$ endowed with nonexpanding operations
$$
\sigma_A \colon A^n \to A \qquad (\sigma\in \Sigma_n)
$$
with respect to the maximum metric on $A^n$. It is \emph{complete} if $A$ is a complete metric space.
\end{defi}

\begin{nota}\label{N:free}
The category of quantitative algebras and nonexpanding homomorphisms is denoted by
$$
\Sigma \mbox{-} \Met\,.
$$
Analogously for its full subcategory
$$
\Sigma \mbox{-} \CMet
$$
of complete quantitative algebras.
\end{nota}

\begin{example}\label{E:term}
\emph{Term algebras.} In universal algebra the free $\Sigma$-algebra on a set $V$ of variables is the algebra $T_\Sigma V$ of terms. Terms are either variables,  or composite terms $\sigma (t_i)_{i<n}$ for $\sigma\in \Sigma_n$ and an $n$-tuple $(t_i)_{i<n}$ of terms. The \emph{depth} $\delta$ of a variable or a constant  is $0$, and the term $s= \sigma (t_i)_{i<n}$ for $n\geq 1$ has depth
$$
\delta(s)= 1+\max_i \delta (t_i)\,.
$$

Analogously for the free quantitative algebra $T_\Sigma X$ on a space $X$: its underlying algebra is the algebra $T_\Sigma |X|$ of terms.  Let us call terms $t$ and $t'$ \emph{similar} if we can obtain $t'$ from $t$ by changing some variables. Thus, all pairs of variables are similar. And terms similar to $\sigma(t_i)$ are precisely the terms $\sigma(t_i')$ with $t_i$ and $t_i'$ similar for each $i$. The metric
$$
d_X^\ast
$$
of the free quantitative algebra $T_\Sigma X$ is defined recursively as follows:
$$
d_X^\ast (t, t') =\begin{cases}
d_X(t, t') & \mbox{if\ \ }t, t' \in X\\
\max\limits_{i<n} d_X^\ast(t_i, t_i') & \mbox{if}\ \ t=\sigma(t_i) \ \mbox{and}\   t'=\sigma(t_i')\\
\infty & \mbox{if}\ \ t \ \mbox{is not similar to} \ t'
\end{cases}
$$
\end{example}

\begin{remark}\label{L:free}
The quantitative algebra $(T_\Sigma X, d_X^\ast)$ is free: for every quantitative algebra $A$ and every nonexpanding function $f\colon X\to A$ there is a unique \emph{extension} to a nonexpanding homomorphism
$$
f^{\#} \colon (T_\Sigma X, d_X^\ast) \to A\,
$$ 
with $f^{\#} \cdot \eta_X =f$. 
Indeed, the classical extension $f^{\#}$ in universal algebra is nonexpanding with respect to the metric $d_X^\ast$.
\end{remark}

For a complete metric space $X$ the situation is analogous: the space $T_\Sigma X$ has the complete metric $d_X^\ast$ above. This is the free algebra on $X$ in $\Sigma$-$\CMet$.

\begin{coro}\label{C:free}
The monad  $T_\Sigma$ of free quantitative algebras on $\Met$ or $\CMet$ is strongly finitary. It preserves epimorphisms.
\end{coro}

%%%%%%%%%%%
Indeed, let $\approx$ be the similarity equivalence on $T_{\Sigma}|X|$. Then $T_\Sigma X$ is the coproduct of spaces $X^n$, one copy for every equivalence class of $\approx$ of terms on precisely $n$ variables. Moreover, this  coproduct is independent of $X$. In other words, the functor $T_\Sigma$ is a coproduct of functors $(-)^n$, one copy of each equivalence class. By (1) and (2) in Example \ref{E:pol}, $T_\Sigma$ is strongly finitary. Preservation of epimorphisms is clear:
given a dense subspace $A$ of $X$, it follows that $A^n$ is dense in $X^n$.

\begin{defi}[\cite{MPP16}]
A (1-basic) \emph{quantitative equation} is an expression $t=_\var t'$, where $t$ and $t'$ are terms in $T_\Sigma V$ for a finite set $V$ of variables, and $\var \geq 0$ is a real number.

A quantitative algebra $A$ \emph{satisfies} this equation provided that every interpretation $f\colon V\to A$ of the  variables fulfils
$$
d_A\big(f^{\#}(t),  f^{\#}(t')\big) \leq \var\,.
$$

A \emph{variety} $\cv$ (aka 1-basic variety) of quantitative algebras is a full subcategory 
of $\Sigma$-$\Met$ specified by a set  of quantitative equations. Thus, a quantitative algebra lies in $\cv$ iff it satisfies each of the given equations.
\end{defi}

We write $t=t'$ in place of $t=_0 t'$, and call such equations \emph{ordinary}.
 In \cite{MPP17} the number $\var$ was assumed to be rational. But this makes no difference: if $\var>0$ is irrational, choose any dicreasing sequence $\var(n)$, $n\in \N$, of rationals converging to $\var$. Then the equation $t=_\var t'$ is equivalent to the set of equations $t=_{\var(n)} t'$ for $n\in \N$.
 
\begin{exs}\label{E:var}
(1) Quantitative monoids.
This is the variety presented by the usual signature (of a binary operation and constant ${e}$) and the usual ordinary equations: $(xy) z = x(yz)$, $xe =x$ and $ex=x$. Observe that by definition of $\Sigma$-$\Met$ this means that multiplication is nonexpanding:
$$
d(xy, x'y') \leq \max \{d(x, x'), d(y, y')\}\,.
$$

(2) \emph{Actions of quantitative monoids.}  
Let us recall that an action of a classical monoid $M$ on a set $X$ is a mapping from $M\times X$ to $X$ (notation $(m, x) \mapsto mx$) whose  curryfication is a monoid homomorphism from $M$ to  the composition monoid $\Set(X,X)$.  Analogously, given a quantitative monoid $M$, its (quantitative) \emph{action} on a metric space $X$ is a nonexpanding homomorphism from $M$ to $[X,X]$ in $\Met$. This is  a monoid  action such that  $d(mx, my) \leq d(x,y)$ for all $(x,y)\in X^2$, and $d(mx, m'x)\leq d(m,m')$ for all $(m,m')\in M\times M$.

Here $\Sigma$ consists of unary operations $m(-)$ for $m\in M$. This variety is presented by the usual ordinary equations:
$$
m(m'x) = (mm')x \quad \mbox{and}\quad  ex=x\,,
$$
together with the following  quantitative equations
$$
mx=_\var m'x \quad \mbox{where}\quad \var =d(m, m')\,.
$$

(3) \emph{Quantitative semilattices.} By a semilattice we mean  a join-semilattice with a bottom. Equivalently, a commutative  and idempotent monoid.

Quantitative semilattices are semilattices acting on a metric space with non-expanding  binary joins. In other words, commutative and idempotent quantitative monoids.

(4) \emph{Small metric spaces}. For the empty signature, $\Sigma$-$\Met$ is simply $\Met$. The quantitative  equation
$$
x=_\var y
$$
presents  all metric spaces of diameter at most $\var$.
\end{exs}

Every variety $ \cv$ is equipped with the obvious forgetful functor $U_{\cv} \colon$ $ \cv \to \Met$.

All the above  has an analogous formulation for complete spaces. A variety $\cv$ is a subcategory of $\Sigma$-$\CMet$ specified by a set  of quantitative equations.

\begin{theorem}[\cite{MPP16}, Sections 6 and 7]
{\rm (1)} In every variety $\cv$ of quantitative algebras each space $X$ generates a free algebra $F_{\cv} X$. That is, the forgetful functor
$$
U_{\cv} \colon \cv \to \Met
$$
has a left adjoint $F_{\cv} \colon \Met \to \cv$.

{\rm (2)} In every variety $\cv$ of complete quantitative algebras each complete space $X$ generates a free algebra $F_{\cv} X$. That is, the forgetful functor
$$
U_{\cv} \colon \cv \to \CMet
$$
has a left adjoint $F_{\cv} \colon \CMet \to \cv$.
\end{theorem}

\begin{nota}
We denote by $T_{\cv} = U_{\cv} F_{\cv}$ the free-algebra monad on $\Met$ or $\CMet$, respectively.
\end{nota}

% stranka 10 velky1

\begin{exs}\label{E:monad}
(1) A free quantitative monoid on a space $X$ in $\Met$ or $\CMet$ is the coproduct (disjoint union) of finite powers. That is, the word monoid 
$$
T_{\cv} X= X^\ast
$$ with the following metric
$$
d(x_0 \dots x_{n-1}, y_0 \dots y_{m-1})=
\begin{cases}
\max\limits_{i<n} d(x_i, y_i) & \mbox{if}\  n=m\\
\infty & 
\mbox{else.}
\end{cases}
$$
This is a strongly finitary monad (\ref {E:var} (1), and (2)).

(2) Free quantitative semilattices in $\CMet$ are given by the Haudorff functor (Example \ref{E:pol}):
$$
T_\cv =\ch\,,
$$
see \cite{BHMW}. This monad is strongly finitary (\cite{ADV}, Example 2.25).

In $\Met$ the free  semilattice  on a space $X$ is the semilattice
$$
T_{\cv} X= \mathcal{H}_f X
$$
of all finite subsets with the Hausdorff metric.

(3) Let $E$ be a metric space (of exceptions). Moggi's exception monad (\cite{M}) is the coproduct
$$
TX = X+E\,.
$$
This is the free-algebra monad for the variety of nullary operations indexed by $E$, presented by the following quantitative equations
$$
e=_\var e' \qquad (e, e' \in E)
$$
where $\var = d(e, e')$.

(4) For small  spaces (Example  \ref{E:var}) presented by $x=_\var y$ the free algebras are given as follows:
$$
T_{\cv} (X,d) = (X, \max\{d, \var\})\,.
$$
\end{exs}

%(1) For quantitative monoids, $T_\cv$ is a coproduct of $(-)^n$, $n\in \N$, see Example \ref{E:pol}, (1) and (2).
%
%(2) Let us verify that $\ch$ is strongly finitary on $\CMet$, using  Proposition \ref{P:comp}. Let $f\colon \ch |X| \to Y$  be a nonexpanding map satisfying the condition \eqref{cond}. The space $\ch|X|$ is the discrete one on all finite subsets of $X$. Let $Z\subseteq \ch X$ be the corresponding subspace of all finite sets with the Hausdorff metric. Then $f\colon Z\to Y$ is nonexpanding. Indeed, let $A, B\in Z$ be sets of distance $d_{\ch}(A,B) =\var$. Then $d(a, B)\leq \var$ for each $a\in A$ and $d(A,b) \leq \var$ for each $b\in B$. Thus $A\times B \subseteq X_\var$  (Notation \ref{N:comp}). We are to prove for all subsets $A\ne B$ that
%$$
%d\big(f(A), f(B)\big) \leq \var\,.
%$$
%This is clear if $A$ or $B$ is empty since $\var =\infty$. Otherwise we have $A\times B \in \ch X_\var$ satisfying   $A= \ch l_\var(A\times B)$ and $B= \ch r_\var (A\times B)$. Thus \eqref{cond} implies the inequality above.
%
%(3) The examples of monoid action, exception monad and small spaces are covered  by Section 4.
%
%\begin{theorem}[\cite{ADV}]\label{T:suff}
%Every strongly finitary monad on  $\Met$ is the free-algebra monad of a variety.
%\end{theorem}
% The converse does not hold, see Section 5.
% 

\begin{remark}\label{P:m}

For every variety $\cv$ the free-algebra functor $F_\cv$ is enriched. Thus, the monad $T_{\cv}$ is  enriched.

Indeed, let morphisms $f$, $g\colon X\to Y$ have distance $d(f,g) = \delta$. We are to prove that  
$$
d\big( F_{\cv} f(s), F_{\cv}g(s)
\big) \leq \delta
$$
 holds for all $s\in F_{\cv} X$. Denote by $S\subseteq |F_{\cv}X|$ the set of all $s$ satisfying the desired inequality. Then $S$ contains the image of $\eta_X \colon X\to F_{\cv} X$: if $s = \eta_X (s_0)$, then
$$
d\big( F_{\cv} f(s), F_{\cv} g(s)\big) = d\big( f(s_0), g(s_0)\big) \leq \delta\,.
$$
And $S$ is closed under all operations: given $\sigma \in \Sigma_n$ and $s_i \in S$, then for $s=\sigma (s_i)_{i<n}$ we get
$$
F_{\cv} f(s) = \sigma_{F_{\cv}X} \big(F_{\cv} f(s_i)\big)\,,
$$
and analogously for $F_{\cv} g(s)$. Since $\sigma_{F{\cv}X}$  is nonexpanding, this implies
$$
d\big( F_{\cv} f(s), F_{\cv} g(s)\big)\leq \max\limits_i d\big( F_{\cv} f(s_i), F_{\cv} g(s_i)\big)\leq \delta\,.
$$
The universal property of $\eta
$ implies that there exists no proper subspace of $F_{\cv}X$ which contains the image of $\eta_X$ and is closed under operations. Thus $S= |F_{\cv}X|$, as claimed.
\end{remark}

\begin{nota}\label{N:term}
(1) For every element $t\in T_{\cv}n$  ($n\in \N$) and every quantitative algebra $A\in \cv$ we denote by $t_A$ the corresponding $n$-ary operation: to an $n$-tuple $f\colon n \to |A|$ it assigns $t_A(f)=  \bar f(t)$ where $\bar f\colon T_{\cv} n \to A$ is the unique nonexpanding  homomorphism extending $u$. 

Every homomorphism $h\colon A\to B$ preserves this $n$-ary operation:
$$
h\big(t_A(x_i)\big)= t_B\big(h(x_i)\big)\,.
$$
This is easy to prove by induction on the depth (Example \ref{E:term}) of $t$.

(2) In particular, for every term $t \in T_\Sigma n$ we have
$$t_A (f) = f^{\#} (t).$$
\end{nota}

The classical Birkhoff Variety Theorem characterizes varieties as classes of algebras
closed under products, subalgebras, and homomorphic images. We have the analogous concepts in $\Sigma$-$\Met$ and $\Sigma$-$\CMet$:

(1) A \emph{product} $\prod\limits_{i\in I} A_i$ of quantitative algebras is the categorical  product  (with the supremum metric). The operations are defined coordinate-wise.

(2) A \emph{subalgebra} of a quantitative algebra $A$ is represented by a subspace closed under operations (in case  of $\Sigma$-$\Met$) and a closed subspace closed under operations  (in case of $\Sigma$-$\CMet$).

(3)  A \emph{homomorphic image} of a quantitative  algebra $A$ is an algebra $B$ for which a nonexpanding  surjective homomorphism $h\colon A \to B$ exists.

The following theorem was stated in \cite{MPP17} without a proof. The first proof was presented in \cite{MU}, B19-20.

\begin{birk}\label{T:BVT}
 A full subcategory of $\Sigma$-$\Met$  is a variety iff it is closed under products, subalgebras, and homomorphic images.
 \end{birk}
 
 \begin{open}
 What is the appropriate variety characterization for varieties in $\Sigma$-$\CMet$? It is easy to prove that varieties are closed under products, closed subalgebras, and images of dense homomorphisms. Does the converse hold?
 \end{open}
 
 \begin{defi}\label{D:con}
 	A \emph{concrete category} over $\Met$  or $\CMet$ is an enriched category together with an enriched functor  $U_\ck$ from $\ck$ to $\Met$ or $\CMet$ which is \emph{faithful}: $d(f,g) = d(U_\ck f, U_\ck g)$ for parallel pairs $f$, $g$.
 	
 	A \emph{concrete functor} $F\colon \ck \to \ck'$ is an enriched functor such that 
 	$$
 	U_\ck = U_{\ck'} \cdot F\,.
 	$$
 \end{defi}
 Every concrete functor is clearly enriched. Examples of concrete categories are varieties, and monadic categories $\Met^T$ or $\CMet^T$.

 \begin{theorem}\label{T:Beck}
 	Every variety $\cv$ of quantitative algebras is concretely isomorphic to the category of algebras for $T_{\cv}$: the comparison functor  $K_\cv$ from $\cv$ to $\Met^{T_{\cv}}$ or $\CMet^{T_\cv}$ is a concrete isomorphism.
 \end{theorem}
 %%%%%%%%%%%
 \begin{proof}
 	The functor $K_\cv$ assigns to every algebra $A$ in $\cv$ the Eilenberg-Moore algebra $\alpha \colon T_\cv A \to A$ given by the unique homomorphism with $\alpha \cdot \eta_A =\id_A$. We see that $K_\cv$ is a concrete (thus enriched) functor.
 	
 	The proof that $K_\cv$ is invertible is analogous to the casse of classical varieties (in $\Set$), see e.g. Theorem  VI.8.1 in \cite{ML}.
 \end{proof}
 
 %\begin{coro}\label{C:Bek}
 %For every variety $\cv$ the forgetful functor $U_\cv $  creates limits.
 %\end{coro}
 
 %\begin{nota}\label{N:mnd}
 % We denote by
 %$$
 %\Mnd_f (\Met)\quad \mbox{or}\quad \Mnd_f(\CMet)
 %$$
 %the enriched category of enriched finitary monads  and monad morphisms, where %the distances are
 %$$
 %d(\varphi, \varphi') =\sup_{X\in \Met} d(\varphi_X, \varphi_X')
 %$$
 %for all parallel monad morphisms $\varphi$, $\varphi'\colon T\to S$.
 %\end{nota}
 %%%%%% 16 a, b
 
 \begin{remark}\label{R:BW}
 	(1) We thus can identify an algebra $A$ of a variety $\cv$ with the corresponding
 	algebra $a\colon T_\cv A\to A$ of $\Met^{T_\cv}$ or $\CMet^{T_\cv}$. Here $a$ is the unique homomorphism of $\cv$ extending $\id_A$.
 	
 	(2) Given a nonexpanding  map $f\colon X\to A$, the homomorphism $\bar f\colon T_\cv X\to A$ extending it is
 	$$
 	\bar f = a\cdot T_\cv f\,.
 	$$
 	Indeed, we have
 	$$ \bar f\cdot \eta_X = a\cdot \eta_A \cdot f =f
 	$$
 	due to $a\cdot \eta_A=\id_A$. And $\bar f$ is a homomorphism since both $a$ and $T_\cv f$ are.
 	
 	(3) Let $k\colon T_\Sigma X \to T_\cv X$ be the unique homomorphism with $\eta_X =k \cdot \eta_X^\Sigma$. Then
 	$$
 	f^{\#} = \bar f \cdot k= a\cdot T_\cv f \cdot k \colon T_\Sigma X\to A\,.
 	$$
 	Indeed, since $\bar f \cdot k\colon T_\Sigma X\to A$ is a homomorphism, we just need to observe that it extends $f$:
 	$$
 	(\bar f\cdot k) \cdot \eta_X^\Sigma = \bar f \cdot \eta_X=f\,.
 	$$
 \end{remark}
 
 \begin{lemma}\label{L:BW}
 	The morphism $k\colon T_\Sigma X \to T_\cv X$ above is an epimorphism.
 \end{lemma}
 
 \begin{proof} 
 	The image of $k$ is a subspace $m\colon M\hookrightarrow T_\cv X$ which, since $k$ is a homomorphism, is closed under the operations.
 	
 	(1) Let $\cv \subseteq \Sigma$-$\Met$.
 	By Theorem \ref{T:BVT} this subalgebra $M$ lies in $\cv$. Let
 	$$
 	\eta'_X \colon X \to M\,, \quad \eta_X = m\cdot \eta'_X\,,
 	$$
 	be the codomain restriction of $\eta_X$. It extends to a homomorphism $e\colon T_\cv X\to M$ with
 	$$
 	(m\cdot e) \cdot \eta_X = m\cdot \eta'_X = \eta_X\,.
 	$$
 	Since $m\cdot e$ is a homomorphism, this implies $m\cdot e =\id$. Therefore, $M= T_\cv X$.

 	(2) Let $\cv \subseteq \Sigma$-$\CMet$. Then the subspace $M$ need  not be complete. Its closure $\bar m \colon \bar M \hookrightarrow T_\cv X$ is a complete subalgebra. Indeed, every operation $\sigma \in \Sigma_{n}$ yields a nonexpanding map $\sigma$ from $M^n$ to $M$ (because  $k$ is a homomorphism). Since $M^n$ is dense in $\bar M^n$, the embedding $M^n \hookrightarrow \bar M^n$ is a Cauchy completion of $M^n$ (Remark \ref{R:Cauchy}). Thus, $\sigma$ extends to a nonexpanding map $\bar \sigma \colon \bar M^n\to \bar M$. We obtain a closed subalgebra $\bar M$ of $T_\cv X$, thus $M$ lies in $\cv$. This implies $\bar M = T_\cv X$, as in Item (1).
 \end{proof}
 
 We now prove that the monads $T_\cv$ on $\Met$ or $\CMet$ are finitary. This follows from the next result.
 
 \begin{prop}\label{L:dir}
 Every variety of (complete) quantitative algebras is closed in $\Sigma$-$\Met$ or $\Sigma$-$\CMet$ under directed colimits.
 \end{prop}

 \begin{proof}
 	Let $D$ be a directed diagram having objects $D_i$ ($i\in I$), with  a colimit cocone $c_i \colon D_i \to C$ ($i\in I$). We prove that every quantitative equation
 	$$
 	t=_\var t'
 	$$
 	holding in $D_i$ for each $i$ also holds in $C$. Our proposition then follows trivially.
 	
 	(1) Let $V$ be the set of variables that appear in $t$ or $t'$.
% 	Every interpretation 
% 	$$
% 	f\colon V\to UC =\colim\limits_{i\in I} UD_i
% 	$$
% 	factorizes, due to Proposition A1, through some $c_i$ ($i\in I$):
% 	$$
% 	\xymatrix@C=2pc@R=2pc{
% 		&
% 		UD_i
% 		\ar[d]^{c_i}\\
% 		V \ar[ur]^{g} \ar[r]_f & UC
% 	}
% 	$$
% 	Since $c_i$ is a homomorphism, we have
% 	$$
% 	f^{\#} = c_i \cdot g^{\#} \colon T_\Sigma V \to C\,.
% 	$$
% 	We know that
% 	$$
% 	d\big(g^{\#} (t), g^{\#} (t')\big) \leq \var
% 	$$
% 	and $c_i$ is nonexpanding. Thus, $d\big(f^{\#} (t), f^{\#} (t')\big) \leq \var$,  as desired.
% 
% (2) The proof for $\CMet$. We again prove $d\big(f^{\#} (t), f^{\#} (t')\big) \leq \var$.
The union $m \colon B  \hookrightarrow C$ of the images of all $c_i$, as a subspace of $C$,
is a subalgebra of $C$, since each $c_i[D_i ]$ is, and the poset $I$ is directed. By Propositions A1 and A2, $B$ is dense in $C$. Therefore, the power $B^V$ is dense in $C^V$,
because $V$ is finite. Thus, there exists a sequence in $B^V$, $g_k \colon V \to B$, such that $f$ is the limit 
of $m \cdot g_k$ for $k \to \infty$.

Let  $\gamma \colon T_\Sigma C  \to C$ represent the algebra C (Remark \ref{R:BW}), then 
$$f^{\#}= \gamma \cdot T_\Sigma f \colon T_\Sigma V \to C.$$
Analogously, when representing $B$ by $\beta \colon T_\Sigma B \to B$, we have $g_k^{\#} = \beta \cdot T_\Sigma g_k$. As $m$ is a homomorphism, we have $m \cdot \beta =\gamma \cdot T_\Sigma    m$. 
%Therefore
%$$m \cdot g_k^{\#} = m \cdot \beta \cdot T_\Sigma g_k = \gamma \cdot T_\Sigma (m \cdot g_k).$$
Since $T_\Sigma$ is enriched, it is locally continuous. This proves that $f^{\#}$ is a limit of the sequence $m \cdot g_k^{\#}$:
$$f^{\#} =\gamma \cdot T_\Sigma (\lim_k  m\cdot g_k) = m\cdot \lim_k ( \beta \cdot T_\Sigma g_k) =m \cdot \lim_k g_k^{\#}.$$

(2) The forgetful functor $U_\Sigma$ to $\Met$ or $\CMet$ preserves directed colimits because they commute  with finite products (Example \ref{E:pol}). Thus for every $k$ there exists $i(k) \in I$ such that the image of $g_k$ lies in the image of $c_{i(k)}$.
We thus have a map $\bar g_k \colon V \to D_{i(k)}$ with $c_{i(k)} \cdot \bar g_k = m \cdot g_k$.We conclude that
$c_{i(k)} \cdot \bar g^{\#}_k = m \cdot g_k^{\#}$: both sides are homomorphisms extending $m \cdot g_k$.
Since $D_{i(k)}$ satisfies the given equation, we have (using that $c_{i(k)}$ is nonexpanding) that
$$d(m\cdot g_k^{\#} (t), m \cdot g_k^{\#} (t')) \leq \var,$$
for every $k$. This proves, via Item (1), the equation $d(f^{\#} (t), f^{\#} (t'))\leq \var.$

\end{proof}

\begin{coro}\label{C:fin}
For every variety $\cv$ the monad $T_{\cv}$ on $\Met$ or $\CMet$ 
 is finitary.
\end{coro}

In fact, the forgetful functor $U_\Sigma$ is finitary (Example 2.6, (1) and (2)). Since $\cv$ is closed under directed colimits, its
forgetful functor is also finitary. It follows that $T_{\cv}$ is finitary.

By now we know that the free-algebra monads of varieties are enriched and finitary. Unforunately, these two properties
do not characterize the monads $T_{\cv}$. But \emph{strongly} finitary monads do
have the form $T_{\cv}$: 

\begin{theorem}\label{T:adv}%3.17
	Every strongly finitary monad $T$ on $\Met$ or $\CMet$
	is the free-algebra monad for a variety of quantitative algebras.
\end{theorem}

\begin{proof} 
Let us choose a countable set $V=\{x_k\}_{k\in \N}$ of variables, and put 
$$
\bm{n} = \{x_0, \dots , x_{n-1}\}\quad \mbox{for}\quad n\in \N\,.
$$
(1)	We define a variety $\cv$ of quantitative $\Sigma$-algebras, where 
	$
	\Sigma_n = |T{\bm n}| \mbox{\ for\ } n\in \N
	$.
	Thus, an $n$-ary symbol $\sigma$ is an element of $T{\bm n}$. We identify $\sigma$ with the term $\sigma(x_0, \dots, x_{n-1})$ in $T_\Sigma \bm{n} $.
% where $V=\{x_i\}_{i\in \N}$ is a chosen set of variables.

	The variety $\cv$ is presented by the ordinary equations (i) and (ii) below  describing the monad structure  $\eta$ and $\mu$ of $T$, together with  equations (iii) describing the metric $d_{\bm{n} }$ of the space $T\bm{n}$:
		\begin{equation*}
	\eta_{\bm n} (x_i)= x_i \qquad (i<n)\,;\tag{i}
	\end{equation*}
		\begin{equation*}
	\mu_{\bm k} \cdot Tf(\sigma) = \sigma \big(f(x_i)\big)_{i<n}\quad  (f\colon {\bm n}\to T{\bm k}\ \mbox{and}\  \sigma \in \Sigma_n)\,;\tag{ii}
	\end{equation*}
\begin{equation*}
	\sigma=_\var \sigma'\quad \mbox{for all }\sigma, \sigma' \in \Sigma_n \mbox{\ with } d_{\bm n}(\sigma, \sigma') =\var \,.\tag{iii}
	\end{equation*}
 
	Here $n$ and $k$ range over $\N$. 
%%%%%%%%%%%%%%%%%%%%%%%%%%%%%%%%%%%%%%%%%%%%%%%%%%%%%%%%%%%%%%%%%%%%%%%%%%%%%%%%%%%%%%%%%%%%%%%%%%%%%%%%%%%%%%%%%%%%%%%%%%%%%%%%%%%%%%%%%%%%%%%%%%%%%%%%%%%%%%%%%%%%%%%%%%%%%%%%%%%%%%%%%%%%%%%%%%%%%%%%%%%%%%%%%%%%%%%%%%%%%%%%%%%%%%%%%%%%%%%%%%%%%%%%%%%%%%%%%%%%%%%%%%%%%%%%%%%%%%%%%%%%%%%%%%%%%%%%%%%%%%%%%%%%%

{\normalsize }(2) For every space $A$ the morphisms $Tf\colon T{\bm n} \to TA$ (for $n\in \N$ and $f\colon {\bm n} \to A$) form a collectively dense  cocone.
Indeed, use that $Ti_A$ is epic (Proposition \ref{P:comp}), and
that, for $f\colon  {\bm n} \to |A|$, all $Tf$ form a colimit cocone (Corollary A3). Therefore, the images of that cocone are collectively epic (that is, dense) by Propositions A1 and A2.

%%%%%%%%%%%%%%%%%%%tady zacina FOSSACS
%%%%%%%%%%%%%%%%%%%%%%%%%%%%%%%%%%%%%%%%%%%%%%
\vskip 1mm %%%%%%%%%%%%%%%%%%%%%%%%%%%%ODTUD
%%%%%%%%%%%%%%%%%%%%%%%%%%%%%%%%%%%%%%%%%%%%%%%%%%%%%%%%%%%%%%%%%%%%%%%%%%%%%%%%%%%%
(3) Every algebra $\alpha \colon TA\to A$ in $\Met^T$ or $\CMet^T$ yields a $\Sigma$-algebra  $R(A, \alpha)$ (shortly $RA$) on the space $A$. Its operation $\sigma_{RA}$, for $\sigma\in \Sigma_n$, is defined as follows
$$
\sigma_{RA} (f) = \alpha \cdot Tf(\sigma) \quad \mbox{for all}\quad f\colon {\bm n} \to |A|\,.
$$
Then $RA$ lies in $\cv_T$. To verify this, we first observe that for every interpretation $v \colon {\bm  n}\to A$ the homomorphism
$v^{\#} \colon T_\Sigma {\bm n} \to A$  restricts on the subset $j_n \colon T{\bm n} \hookrightarrow T_\Sigma {\bm n}$ (of terms of depth $1$) to the composite of $Tv$ with $\alpha$:
\begin{equation*}
v^{\#} \cdot j_{\bm n} = \alpha \cdot Tv \colon T\bm{n}  \to A\,. \tag{*}
\end{equation*}
This follows from the definition of the operations of $RA$:
given $\sigma \in \Sigma_n$, we have (since $\sigma$ represents $\sigma(x_0, \dots , x_{n-1})$ that
$$
v^{\#} \cdot j_{\bm n}(\sigma) = v^{\#} \big(\sigma(x_i)_{i<n}\big) = \sigma_A \big(v(x_i)\big)_{i<n} = \alpha \cdot T v(\sigma)\,.
$$
 We now verify that $RA$ satisfies (i) - (iii) above.

(i) The equality $v^{\#} (\eta _{\bm n} (x_i) ) =v(x_i)= v^{\#}(x_i)$ follows from $\alpha \cdot \eta_A =\id$, using \thetag{*}.

(ii) We verify 
$$
v^{\#} (\mu_{\bm k} \cdot Tf(\sigma) )= v^{\#} (\sigma_{RA} (x_i) )).
$$

This follows, as $\alpha \cdot T\alpha = \alpha \cdot \mu_A$, from the commutative diagram below:
$$
\xymatrix@C=2pc{
	T\bm{n} \ar@{-->}[r]^{Tf} & 
	T^2{\bm k}  \ar[d]_{\mu_{\bm k}} \ar[r]^{T^2 v} & 
	T^2A\ar@<-.6ex>[d]_{T\alpha} \ar@<.6ex>[d]^{\mu_A}\\
	&T{\bm k} \ar[r]_{Tv} \ar[d]_{j_{\bm k}} & TA \ar[d]^{\alpha}\\
	& T_\Sigma {\bm k} \ar[r]_{v^{\#}}& A
}
$$	
Using \thetag{*}, applied to $Tf(\sigma)$ and $\alpha \cdot \mu_A = \alpha \cdot T\alpha$, the left-hand side is equal to
$$
v^{\#} (\mu_{\bm k}  \cdot j_{\bm k} \cdot Tf(\sigma) )=\alpha \cdot T\alpha \cdot T^2 v \cdot Tf(\sigma) = \sigma_{RA} (\alpha \cdot Tv \cdot f)\,.
$$
This is the same as  to the right-hand side:
$$v^{\#} (\sigma(f(x_i))=\sigma_{RA} (v^{\#}\cdot j_k \cdot f(x_i))= \sigma_{RA} (\alpha \cdot Tv\cdot f(x_i)).$$

(iii) Since $v^{\#} \cdot j_{\bm n}$ is nonexpanding, we conclude that $d_{\bm n} (\sigma, \sigma ')< \var$ implies $d(v^{\#} (\sigma), v^{\#} (\sigma ')) \leq \var$.

\vskip 1mm
(3) The homomorphisms $h$ from $(A, \alpha)$ to $(B, \beta)$ in $\Met^T$ are precisely the homomorphisms from $RA$ to $RB$ in $\cv_T$. Indeed, assume first $h \cdot \alpha = \beta \cdot Th$. For every interpretation $f\colon {\bm n} \to A$ we get that
$h(\sigma_{RA} (f)) =\sigma_{RB} (h\cdot f),$
since the left-hand side is $h \cdot \alpha \cdot Tf(\sigma)$, and the right-hand one is $\beta \cdot Th \cdot Tf (\sigma)$.

Conversely, if $h$ is a $\Sigma$-homomorphism, we prove for all interpretations $f \colon {\bm n} \to A$, that 
$ (h \cdot \alpha) \cdot Tf = (\beta \cdot Th) \cdot Tf$. (This concludes the proof by Item (2).)  The left-hand side, applied to $\sigma \in T{\bm n}$, yields $h(\sigma _A(f)) = \sigma_ B (h \cdot f)$.
% 	We denote by $E\colon \Met^T\to \cv_T$ the resulting concrete full embedding. Notice that for $(TX, \mu_X)$ we have $\sigma_{E(TX)} (\eta_X) =\sigma$.
Which is precisely the right-hand side applied to $\sigma$.	

\vskip 1mm
(4) We thus get a concrete full embedding $R \colon \Met^T \to \cv_T$. To prove that $T$ is the free-algebra monad of $\cv$, it is sufficient, due to Theorem \ref{T:Beck},  to verify  for all metric spaces $X$ that $R(TX, \mu_X)$ is the free algebra of $\cv_T$ with the universal map $\eta_X\colon X \to TX$.

(4a) Let $X$ be finite and discrete. Say, $X={\bm n}$\ ($n\in \N$). Given an algebra $A\in \cv_T$ and an interpretation $v\colon {\bm n}\to |A|$, we verify that the morphism $\bar v\colon T{\bm n} \to A$ assigning to $\sigma \in |T{\bm n}|=\Sigma_n$ the value
$$
\bar v (\sigma) =\sigma_A(v)
$$
is a nonexpanding homomorphism $\bar v \colon R(T{\bm n}, \mu_{{\bm n}}) \to A$ with $v=\bar v \cdot \eta_{{\bm n}}$, and it is unique.

To show that $\bar v$ is nonexpanding, we first recall that 
$$v^{\#}(\sigma) = \sigma_A\big( v(x_0),\dots , v(x_{n-1})\big) =\bar v(\sigma)$$ 
for all $\sigma \in Tn$. Let $\sigma$, $\sigma'\in T{\bm n}$ have distance $\var$, then equations (iii) imply $d_A\big(v^{\#}(\sigma), v^{\#}(\sigma')\big) \leq \var$. Thus $d_A\big( \bar v(\sigma), \bar v(\sigma')\big)\leq \var$.
 
 To prove that $\bar v$ is a $\Sigma$-homomorphism, consider a $k$-ary operation $\sigma\in T{\bm k}$ and a $k$-tuple $f\colon {\bm k}\to |RT{\bm n}|=\Sigma_n$. We verify that
 $$
  \bar v \big( \sigma_{R(T{\bm n})} (f)\big) = \sigma_A(\bar v \cdot f)\,.\
 $$
 To compute the left-hand side, $l$, put $\tau=\sigma_{R(T{\bm n} )} (f) \in \Sigma_n$: 
  $$
l= \bar v \big(\sigma_{R(T{\bm n})}(f)\big) = \bar v (\tau) = \tau_A(v) = 
 v^{\#}(\tau)\,.
 $$
 The operations $\sigma$ of the algebra  $R(Tn, \mu_{\bm n})$ are defined by
 $$
 \tau = \sigma_{R(T{\bm n})} (f) = \mu_{\bm n} \cdot Tf(\sigma)\,.
 $$
 Thus,
 $$
 l=v^{\#}(\tau) = v^{\#}\big(\mu_{\bm n} \cdot Tf(\sigma)\big)\,.
 $$
 Since $A$ satisfies equations (ii), we conclude that
 $$
 \bar v\big(\sigma_{R(T{\bm n})} (f) = v^{\#}(\tau) = v^{\#}\big(\sigma(f(x_0), \dots , f(x_{k-1})\big)\,.
 $$
 For the $n$-ary operations $\varrho^i = f(x_i) \in \Sigma_n$ this last result is $v^{\#}$ applied to the term $\sigma(\varrho^0, \dots , \varrho^{k-1})$, which yields $\sigma_A \big(\varrho_A^0(v), \dots , \varrho_A^{k-1}(v)\big)$. Thus, from $\varrho^i_A(v) = \bar v (\varrho^i) = \bar v \cdot f(x_i)$ we get that
 $$
 l=\bar v \big( \sigma_{R(T{\bm n})}(f)\big) =\sigma_A \big(\bar v \cdot f(x_0), \dots , \bar v \cdot f(x_{n-1})\big)\,,
 $$
 as required.
 
 The equality $v= \bar v \cdot \eta_{\bm n}$ follows from equations (i): we have
 $$
 v(x_i) = v^{\#}\big(\eta_{\bm n}(x_i)\big) = \bar v \big(\eta_{{\bm n}}(x_i)\big) \quad \mbox{for\ \ } i<n\,.
 $$

To prove uniqueness, let $h\colon R(T{\bm n}, \mu_{{\bm n}}) \to A$ be a homomorphism with $h\cdot \eta_{\bm n} =v$. We prove $h\cdot Ti_{{\bm n}} = \bar v\cdot Ti_{\bm n}$, and  apply Proposition \ref{P:comp}.b to conclude $h=\bar v$. For every $\sigma \in |T{\bm n}|$, since  the operations of 
$R(T {\bm n}, \mu_{\bm n})$ yield $\sigma_{R(T{\bm n})} (\eta_{\bm n})= \mu _{\bm n} \cdot T\eta _{\bm n} (\sigma)= \sigma$,
and $h$ preserves $\sigma_{R(T{\bm n})}$, we get $ h(\sigma)= \sigma_A(h\cdot \eta_{\bm n}) = \sigma_A(v) = \bar v(\sigma)$.

\vskip 1mm

(4b) For an arbitrary finite space $X$, we use that  it is a weighted colimit $X=\colim\limits_B D_X$ (Remark \ref{R:X}), and $TX =\colim\limits_B TD_X$ (Proposition \ref{P:X}). All spaces $Y$ in the image of $D_X$ are finite and discrete, thus $R(TY, \mu_Y)$ is free on $Y$ in $\cv_T$ by Item (4a). Since the free-algebra functor preserves weighted colimits, being an enriched  left adjoint, we conclude that $R(TX, \mu_X)$ is free on $X$ in $\cv_T$.

\vskip 1mm

(4c) For an arbitrary space $X$, we have the directed colimit $X=\colim\limits_{i\in I} X_i$ where $X_i$ ranges over all finite subspaces (Corollary A3). % (Example \ref{E:dir}).
 Since $T$ is finitary,  $TX=\colim TX_i$, and from Item (4b) we conclude that $R(TX, \mu_X)$ is free on $X$ in $\cv_T$.
\end{proof}

%%%%%%%%%%%%%%%%%%zpet
\section{The Category of Varieties}
We introduce the category of varieties, and describe weighted limits in it. This is used in the next section for a characterization  the monads $T_\cv$ for varieties $\cv$.
 Recall concrete funcors from Definition \ref{D:con}.

\begin{defi}[Category of varieties]\label{D:var}
	We denote by
	$$
	\cv ar (\Met)
	$$
	the  category of all varieties of quantitative algebras (for arbitrary signatures).
	Morphisms from $\cv$ to $\cw$ are the  {concrete functors} $G\colon \cv\to \cw$
	$$
	\xymatrix@C=1pc@R=2pc{
		\cv \ar [rr]^{G} \ar[dr]_{U_{\cv}} && \cw\ar[dl]^{U_{\cw}} \\
		&\Met &
	}
	$$
The distance of morphisms $G$, $G'\colon \cv \to \cw$ with $\cw \subseteq \Sigma$-$\Met$ is
	$$
	d(G, G') = \sup  d(t_{GA}, t_{G'A}),
	$$
where $t$ ranges over all terms $T_\Sigma n$, and $A$ over algebras of $\cv$.
\end{defi}	
	
	To verify that $d$ is a metric, and $\cv ar(\Met)$ is indeed enriched, we use Proposition \ref{L:faith}
	below.

\begin{prop}[\cite{BW}, Theorem 3.6.3]\label{P:BW}
Given monads $T$ and $S$ on $\Met$, there is  a bijective correspondence between monad morphisms $\gamma \colon T\to S$ and concrete functors $G\colon \Met^S \to \Met^T$. To every algebra $\alpha \colon SA \to A$ the functor $G$ assigns the algebra
$$
TA \xrightarrow{\ \gamma_A\ } SA \xrightarrow{\ \alpha\ } A \quad \mbox{in\quad}\Met^T\,.
$$
Analogously for monads on $\CMet$.
\end{prop}
%%%%%%%%%%%%%%%%%%%%%%%%%%%%%%%%%%%%%%%%%%%%%%%%%%%%%%%%%%%%%%%%%%%%%%%%%%%%%%%%%%16a-16d
\begin{example}\label{E:BW}
(1) The embedding of a variety $\cv$ is a concrete functor $\cv \hookrightarrow \Sigma$-$\Met$. The corresponding \emph{canonical monad morphism}
$$
k_\cv \colon T_\Sigma \to T_\cv
$$
has the components $(k_\cv)_X\colon T_\Sigma X \to T_\cv X$ of Lemma \ref{L:BW}.

(2) Let $\Gamma$ be the signature of a single $n$-ary operation $\gamma$. For every term $t\in T_\Sigma {\bm n}$ we have a concrete functor
$$
F_t\colon \Sigma\mbox{-}\Met \to \Gamma\mbox{-}\Met\,.
$$
It assigns to every $\Sigma$-algebra $A$ the $\Gamma$-algebra $\gamma_A\colon A^n \to A$ defined by computing $t$ in $A$: $\gamma_A (a_i)_{i<n} = t_A (a_i)_{i<n}$ (Notation \ref{N:term}).

The corresponding monad morphism
$$
\widehat t\colon T_\Gamma \to T_\Sigma
$$
is the substitution: every term in $T_\Gamma X$ is turned to a term in $T_\Sigma X$ by substituting every occurence of $\gamma$ by the term $t$. More precisely, the $X$-component assigns to every term $s \in T_\Gamma X$ the following value:
$$
\widehat t_X (s) =\begin{cases}
s & \mbox{if}\quad s\in X\\
t\big(\widehat t_X(s_i)\big)_{i<n} & \mbox{if}\quad s= \gamma(s_i)_{i<n}\,.
\end{cases}
$$
Indeed, for every $\Sigma$-algebra $A$ expressed via the Eilenberg-Moore algebra $\alpha \colon T_\Sigma A \to A$ the map $\alpha_X$ takes a term $t$ and computes it in $A$, with the interpretation of variables $\id_A$. The Eilenberg-Moore algebra  $\alpha \cdot \widehat t_A \colon T_\Gamma A \to A$ is thus the computation  of the term $t$, as claimed.

%(3) More generally, let $\Sigma$ and $\Gamma$ be arbitrary signatures, and let $\cv \subseteq \Sigma$-$\Met$ be a variety. A concrete  functor $P\colon \cv \to \Gamma$-$\Met$ is given by the following choice map $\varphi_P$. It chooses for every symbol $\gamma \in  \Gamma$ of arity $n$ an element $\varphi_P(\gamma)$ of $T_\cv n$. 
 	
 %	The functor $P$ assigns to an algebra $A\in \cv$ the $\Gamma$-algebra $PA$ with the operations $\gamma_{PA}$ ($\gamma \in \Gamma_n$) below:
% 	$$
% 	\gamma_{PA} (f) = \bar f\big(\varphi_P (\gamma)\big)\quad \mbox{for all} \ f\colon n\to |A|\,.
 %	$$
 %	Here $\bar f\colon T_\cv n \to A$ is the homomorphic extension of $f$.
 	
 %	Indeed, as proved in \cite{TAKR}, the monad $T_\Gamma$ is free on the endofunctor
 %	$$
 %	H_\Gamma = \coprod_{n\in \N} \Met (\Gamma_n, -)\,.

\end{example}

%\begin{remark}\label{R:sur}
%(1) For every  variety $\cv \subseteq \Sigma$-$\Met$ the components of $k_\cv\colon T_\Sigma \to T_\cv$ are surjective, see Lemma \ref{L:BW}.
%
%%: the image of $(k_\cv)_X$ is a subspace of $T_\cv X$ closed under the operations and containing $\eta_X[X]$. No proper subalgebra of $T_\cv X$ has these properties.
%
%(2) Analogously for $\CMet$: The components of $k_\cv$ are  dense.

%% 18 a,b
%\end{remark}

 % In order to introduce semi-strongly finitary monads in the next section, we
%use the category of finitary monads on $\Met$, where for parallel monad morphisms $\varphi, \psi \colon S \to T$ their distance is defined as the supremum of $d(\varphi _n ,\psi _n)$ over all natural numbers $n$. This distance forms a pseudometric: even if $\varphi$ and $\psi$ are distinct, their distance can be $0$ (since the monads are finitary, but not necessarily strongly finitary). We thus need to work with categories enriched over pseudometric spaces:	
 \begin{lemma}\label{L:epi}
 	The monad $T_\cv$ preserves epimorphisms for every variety $\cv$.
 \end{lemma}

\begin{proof}
The monad $T_\Sigma$ has this property by Corollary \ref{C:free}. Let $k_\cv \colon T_\Sigma \to T_{\cv}$ be the above monad morphism.
	Given an epimorphism $f\colon X\to Y$, from the naturality square
	$$
	(k_\cv )_Y \cdot T_\Sigma e = T_\cv e \cdot (k_\cv)_X
	$$
	we deduce, since the left-hand side is epic by Lemma \ref{L:BW}, that $T_\cv e$ is also epic.
\end{proof}

 \begin{nota}\label{N:mnd}
   We denote by
 $$
 \Mnd_f (\Met)
 $$ the category of enriched finitary monads on $\Met$ and monad morphisms. It is enriched (as a full subcategory of $\Mnd(\Met))$ by the supremum metric.
 
 Anaologously we use $\Mnd_f(\CMet)$.
 \end{nota}

%%%%%%%%%%%%%%%%%%%%%%%%%%%%%%%%%%%%%%%%%%%%%%%%%%%%%%%%%%%%%%%%%%%%%%%%%%%%%%%%%%%%

\begin{prop}\label{L:faith}
Let $G, G'\colon \cv \to \bar\cv$ be concrete functors with the corresponding monad morphisms
 $\gamma, \gamma' \colon T_{\bar \cv} \to T_\cv$. Then
$$
d(G, G') = d(\gamma, \gamma')\,.
$$
\end{prop}

\begin{proof}
We denote the monads $T_\cv$ and $T_{\bar \cv}$ by $(T, \mu, \eta)$ and $(\bar T, \bar \mu, \bar \eta)$, respectively. The signature of $\bar \cv$ is denoted by $\Sigma$.

(1) Let $A$ be an algebra in $\cv$ represented by $a \colon TA \to A$.Then
$GA$ is represented by $a \cdot \gamma_A \colon \bar T A \to A$. For every
$n$-tuple $f \colon n \to |A| = |GA|$ we have, by Remark \ref {R:BW},
$$
f^{\#} = a \cdot \gamma _A \cdot \bar T f \cdot k \colon T_\Sigma n \to GA.
$$
Since $f^{\#} (t) = t_{GA} (f)$  for every term  $t \in T_\Sigma n$ (Notation \ref{N:term}), we have 
$$t_{GA} (f) = a \cdot \gamma_A \cdot \bar T f \cdot k (t).$$
Analogously $t_{G'A} (f)= a \cdot \gamma'_A \cdot \bar T f \cdot k (t).$
Consequently,
$$d(t_{GA} (f), t_{G'A} (f)) \leq d(\gamma_A, \gamma'_A).$$
This proves that $$d(t_{GA},, t'_{GA}) \leq d(\gamma_A ,\gamma'_A).$$

%(1) We first verify that
%$$d(G,G')=\sup _{\sigma \in \Sigma} d(\sigma _{GB} (\eta _n),\sigma _{G'B} (\eta	_n)),$$ 
%where $n$ is the arity of $\sigma$, and $B=Tn$. For that we
%just need to show that, given an algebra $A$ in $\cv$ and an $n$-tple $f
%\colon n \to |A|$ in it, then
%$$d(\sigma_{GA} (f) , \sigma_{G'A} (f)) \leq d(\sigma _{GB} (\eta _n),\sigma _{G'B} (\eta_n)). $$
%Indeed, the unique extension $\bar f \colon Tn\to A$ of $f$, being a
%homomorphism, fulfils
%$$\sigma_{GA} (f) = \sigma _{GA} (\bar f \cdot \eta_n)= \bar f (\sigma _{GB} (\eta _n)).$$
%Analogously for $G'A$. Since $\bar f$ is nonexpanding, this proves the
%desired inequality.

(2) We also have
$$d(\gamma, \gamma ') = \sup _{n \in \Bbb{N}} d(\gamma _n, \gamma ' _n).$$
Indeed, the monad $\bar T$ is finitary  (Corollary \ref{C:fin}), thus $d(\gamma, \gamma ')$ is
the supremum of the distances od $\gamma _M$ and $\gamma '_M$ over all
finite spaces $M$: see Corollary A3. We can assume that the underlying set of $M$ is some natural
number $n$, and we use the fact that
$\bar T i_M \colon \bar T n \to \bar T M$
is dense (Lemma \ref{L:epi}). Thus,
$$d(\gamma _M, \gamma '_M)=d(\gamma _M \cdot \bar T i_M,\gamma '_M\cdot \bar T i_M).$$
By naturality of $\gamma$ and $\gamma '$ this yields the desired inequality:
$$d(\gamma _M, \gamma '_M)=d(\bar T i_M \cdot \gamma _n, \bar T i_M\cdot \gamma '_n) \leq d(\gamma _n, \gamma '_n).$$

%(3)The free algebra
%$$
%B= (Tm, \mu_m) \quad \mbox{in\ }\ \cv
%$$
%is mapped by $G$ to
%$$
%GB \equiv \bar T Tm \xrightarrow{\gamma_{Tm}} TTm \xrightarrow{\mu_m} Tm.
%$$
%
%Let $\Sigma$ be the signature of $\bar \cv$ and $k \colon T_\Sigma n \to
%\bar T n$ the morphism of Remark \ref{R:BW} (3).
%For every $n$-tuple $f\colon n\to |B| = |GB|$, we have, by that Remark, that $f^{\#} \colon T_\Sigma n\to GB$ is given by
%$$
%f^{\#} = (\mu_m \cdot \gamma_{Tm}) \cdot \bar T f\cdot k = \mu_m \cdot Tf \cdot \gamma_n\cdot k\,.
%$$
%Analogously for $G'B$. Thus, for every term in $T_\Sigma$ Notation \ref{N:term} yields
%$$
%t_{GB} (f) = \mu_m\cdot Tf \cdot \gamma_n \cdot k(t)\quad \mbox{and}\quad t_{G'B} (f) =  \mu_m \cdot Tf\cdot \gamma'_n \cdot k(t)\,.
%$$
%Consequently,
%$$
%d\big(t_{GA}(f), t_{G'A}(f)\big) \leq d(\gamma_n, \gamma'_n)\,.
%$$
%Since this holds for all $f$, we have  proved
%$$
%d(G, G') = \sup d\big(t_{GA}(f), t_{G'A}(f)\big)\leq d(\gamma, \gamma')\,.
%$$
%
%To prove the reverse inequality, we verify
%$$
%d(\gamma_n, \gamma'_n) \leq d(G, G') \quad \mbox{for all \ $n\in \N $}.
%$$
%We apply the above to
%$$
%m=n \quad \mbox{and} \ \ f= \eta_n \colon n\to T n
%$$
%to get
%$$
%t_{GA}(f) = \mu_n \cdot T\eta_n \cdot \gamma_n \cdot k(t) = \gamma_n \cdot k(t)\,,
%$$
%and analogously for $t_{G'A}(f) $. Recall that $k$ is dense (Lemma \ref{L:BW}). Therefore, 
%we get
%$$
%d(\gamma _n, \gamma' _n)= \sup _t d(\gamma _n (k(t)), \gamma' _n (k(t))) \leq \sup _td(t_{GA}, t_{G'A}).
%$$
%This proves $d(\gamma_n, \gamma'_n)\leq d(G, G')$, as desired.
(3) Item (1) holds for all terms $t$, thus we have $$d(G,G') \leq
d(\gamma, \gamma').$$ 

To prove the oposite inequality, we apply (1) to the
free algebra $A= (Tn, \mu _n)$ and to $f=\eta_n$. We get from $\mu _n \cdot T
\eta_n = id$ that
$$f^{\#} = \mu _n \cdot \gamma _{\bar T \eta_n} \cdot \bar T \eta _n \cdot k 
= \mu _n \cdot T\eta_n \cdot \gamma _n \cdot k = \gamma _n \cdot k.$$ 
Therefore, $t_{GA} (f)=
\gamma_n \cdot k(t)$. Analogously, $t'_{GA} (f)= \gamma_n \cdot k(t')$. This
proves, since $k$ is dense (Lemma \ref{L:BW}),
that
$$
d(\gamma _n, \gamma'_n) = \sup _t d(\gamma _n (k(t)), \gamma'_n (k(t)) \leq
\sup _t d(t_{GA} (f), t'_{GA} (f))\leq d(G, G').
	$$
	Thus the proof is concluded by Item (2).

\end{proof}
 Let us recall from \cite{K} that an enriched functor $P$ is fully faithful iff all of its restrictions
 $P_{X, Y} \colon  [X, Y] \to [PX, PY]$ are metric isomorphisms.

 %The distance of concrete functors $H, H' \colon \cv \to \cw$ is
%$$
%d(H, H') = \sup d(t_{HA}, t_{H'A})
%$$
%where the supremum ranges over free algebras $A = T_{\cv} n$ ($n\in \N$), and terms $t\in T_{\cw} k$ $(k\in \N)$, see Notation \ref{N:term}.

%%%%%%%%%%%%%%%%%%%%%%%%%%% 
%Recall the enriched category of finitary monads from Notation \ref{N:mnd}.
%An enriched functor $H$ is \emph{faithful} if for all parallel pairs $f$, $f'$ we have $d(f, f') = d(Hf, Hf')$.

\begin{lemma}\label{T:ff}
The following defines an enriched, fully faithful functor
$$
\Phi\colon \cv ar (\Met)^{\op} \to \Mnd_f(\Met)\,.
$$
It assigns to every variety $\cv$ the monad $T_\cv$.
Given a concrete functor $H\colon \cv \to \cw$, we form the following (concrete) composite
$$
\Met^{T_\cv} \xrightarrow{\ K_\cv^{-1}\ } \cv \xrightarrow{\ H\ } \cw \xrightarrow{\ K_{\cw}\ }  \Met^{T_{\cw}}\,.
$$
Then $\Phi(H)\colon T_\cw \to T_\cv$ is the corresponding monad morphism.
\end{lemma}

%%%%%%%%%%%%%%%%%%%%%%
\begin{proof}
%The ordinary functor $E_U$ is well defined (Proposition \ref{P:m}) and fully faithful. 
The monad $\Phi\cv$ is enriched and  finitary (Remark \ref{P:m} and Corollary \ref{C:fin}). The functor $\Phi$ is well defined: it clearly preserves  identity morphisms and composition. Due to the bijection between monad morphisms and concrete functors, $\Phi$ is full.
It is faithful due to Proposition \ref{L:faith}.
\end{proof}
%It remains to prove,  for concrete functors $H, H'\colon \cv\to \cw$, that
%$$
%d(H, H') = d(EH, EH')\,.
%$$
%Since $K_\cv$ and $K_\cw$ are concrete isomorphisms, it is sufficient to prove %this for concrete endofunctors
%$$
%H, H' \colon \Met^{T} \to \Met^{S}
%$$
%for all finitary enriched monads $T$ and $ S$.

%% strana 19 odpada? je comment na konci

%%%%%% 

Analogously we define the full and faithful enriched functor $$\Phi \colon \cv ar (\CMet)^{\op} \to \Mnd_f(\CMet).$$

\vskip 2mm
Recall that an enriched  category $\ck$ is \emph{complete} if it has weighted limits. Equivalently: it has ordinary limits and  {cotensors $M\pitchfork K$ (Example \ref{E:eq}) for spaces $M\in \Met$ or $M\in \CMet$, and objects $K\in \ck$. See \cite{B}, Theorem 6.6.14.

Dually, an enriched category is \emph{ cocomplete} iff it has ordinary colimits  and  tensors. 
An enriched  functor preserves weighted colimits iff it   preserves ordinary colimits and tensors  
(\cite{B}, Corollary 6.6.15).

We next prove that $\cv ar(\Met)$ is complete. First, we describe $\var$-equalizers (Example \ref{E:eq}), since they play a special role below:
%%%%%%%%%       21 a-d
\begin{nota}
	The variety presented by a set $\ce$ of quantitative equation is denoted  by $(\Sigma, \ce)$-$\Met$ or $(\Sigma, \ce)$-$\CMet$. 
\end{nota}

\begin{prop}\label{P:e}
Let $G$, $G'\colon \Gamma$-$\Met\to \Sigma$-$\Met$ be concrete functors and $\gamma$, $\gamma'\colon T_\Sigma \to T_\Gamma$ the corresponding monad morphisms. The $\var$-equalizer of $G$ and $G'$ is the embedding
$$ 
\xymatrix@1{
\cv \overset{I}{\hookrightarrow}\Gamma\mbox{-}\Met \ar@<.6ex>[r]^{\quad G'} \ar@<-.6ex>[r]_{\quad G}& \Sigma\mbox{-}\Met
}
$$
of the variety $\cv$ presented by the following set $\ce_0$ of equations:
$$
\gamma_n(t) =_\var \gamma'_n(t) \quad \mbox{for}\quad n\in \N \quad \mbox{and}\quad t\in T_\Sigma  n\,.
$$
\end{prop}

\begin{proof}
(1) We verify $d(GI, G'I) \leq \var$. Given an algebra $a\colon T_\Gamma A \to A$ in $\Met^{T_\Gamma}$ satisfying $\ce_0$, we are to prove
$$
d(t_{GA}, t_{G'A})\leq \var \mbox{\quad for\quad} t \in T_\Sigma n, \,.
$$
For every $n$-tuple $f\colon n \to A$ the homomorphism $f^{\#}\colon T_\Gamma n\to A$ is given by
$$
f^{\#} = a\cdot T_\Gamma f\,,
$$
see Remark \ref{R:BW} (2). 
Since $\gamma$ is a monad morphism, the composite
$$
f^{\#} \cdot \gamma_n \colon T_\Sigma n \to GA
$$
is a $\Sigma$-homomorphism extending  $f$. Since $GA$ is given by $a \cdot \gamma  _A\colon T_\Sigma  A \to A$,  we get $f^{\#} \cdot \gamma_n =a \cdot \gamma_A \cdot T_\Sigma f$ (by the same remark). We thus have
$$
t_{GA} (f) = f^{\#} \big(\gamma_n (t)\big) = a\cdot \gamma_A \cdot T_\Sigma f(t):
$$
%be the following commutative diagram
%%%%%%%%%%%%%%%%%%%%%%%%%%%%%%%%%%%%%%%%%%%%%%%%%%%%%%%%%%%%%%%%%%%%%%%%%%%%%%%%%%%%%%%%%%%%%%%%%%%%%%
$$
\xymatrix{
T_\Sigma n \ar[r]^{\gamma_n} \ar[d]_{T_\Sigma f} & T_\Gamma n \ar[d]^{T_\Gamma f} \ar[dr]^{f^{\#}}&\\
T_\Sigma A \ar[r]_{\gamma_A} & T_\Gamma A \ar[r]_a&A
}
$$ 
Analogously for $t_{G'A}(f)$. Since $(A, a)$ satisfies $\ce_0$, this proves $d\big(t_{GA}(f)$,  $t_{G'A}(f)\big) \leq \var$, as required.

(2) To prove the universal property, let $J\colon \cw\to \Sigma$-$\Met$ be a morphism of $\cv ar(\Met)$ with
$$
d(GJ, G'J) \leq \var\,.
$$
We are to prove that $J$ factorizes through $I$; this clearly implies the universal property. Let $j\colon T_\Sigma \to T_\cw$ be the monad morphism corresponding to $J$. The proof will be concluded by showing, for every algebra $a\colon T_\cw A \to A$ of $\cw \simeq \Met^{T_\cw}$, that $\ce_0$ holds in its image by $J$:
$$
JA \colon T_\Sigma A\xrightarrow{j_A} T_\cw A \xrightarrow{a} A\,.
$$
 That is, for all $t\in T_\Sigma n$ and $f\colon n\to A$ the homomorphism $f^{\#} \colon T_\Gamma n \to JA$ fulfils
$$
d\big( f^{\#} \cdot \gamma_n(t), f^{\#} \cdot \gamma_n'(t)\big) \leq \var\,.
$$
Using $d(GJ, G'J)\leq \var$, we get that
$$
d\big(t_{GJA}(f), t_{G'JA} (f)\big) \leq \var\,.
$$
The last inequality states that 
$$
d\big( f^{\#} \cdot \gamma_n(t), f^{\#} \cdot \gamma_n' (t)\big) \leq \var\,.
$$
Since this is true for all $t$ and $f$, we see that $JA$ satisfies $\ce_0$.
\end{proof}

\begin{remark}\label{R:e}
(1) The above argument concerning universality works for every enriched monad $S$, not necessarily of the form $T_\cw$: let $J\colon \Met^S$ $ \to \Gamma$-$\Met$ be a concrete functor with $d(GJ, G'J)\leq \var$. Then $J$ factorizes through $I$. Indeed, for every algebra $a\colon JA\to A$ in $\Met^S$ the algebra $JA = (A, a\cdot j_A)$ satisfies $\ce_0$.

(2) Let $\cw$ be a variety of quantitative $\Gamma$-algebras presented by a set $\ce$ of equations. Given concrete functors $G, G'\colon \cw \to \Sigma$-$\Met$,  their $\var$-equalizer is the embedding $I\colon \cv\hookrightarrow \cw$ of the variety of $\Gamma$-algebras presented  by $\ce \cup \ce_0$. The proof is the same as for $\cv =\Gamma$-$\Met$ above.

(3) All of the above results also hold for the category $\cv ar(\CMet)$.
\end{remark}

%%%%%%%%%%%%%%%%%%%%%%%%%%%%%%%%%%%%%%%%%%%%%%%%%%%%%%%

\begin{theorem}\label{T:cp}
The category of varieties is complete, and the functor $\Phi\colon \cv ar(\Met)^{\op} \to \Mnd_f(\Met)$ preserves weighted colimits.
\end{theorem}

\begin{proof}
We prove that $\cv ar(\Met)$ has 
 products and equalizers as well as cotensors, and $\Phi^{\op}$ preserves all these. 
 
 (1) Products of varieties $\cv^i = (\Sigma^i, \mathcal{E}^i)$-$\Met$ for $i\in I$. Let $\Sigma$ be the signature which is a disjoint union of $\Sigma^i$ for $i\in I$.  Thus, every term $t$ for $\Sigma^i$ is also a term for $\Sigma$. 
Moreover, for every $\Sigma ^i$-algebra $A$ the value $t_A$ (Notation \ref{N:term}) is independent of the choice $\Sigma$ or $\Sigma^i$ as our signature. This follows by an easy induction in the depth (Example \ref{E:term}) of $t$.
Then
 $$
 \cw = (\Sigma, \ce)\mbox{-}\Met \quad \mbox{for}\quad \ce =\bigcup_{i\in I} \ce^i
 $$
 is the product of $\cv^i$ in $\cv ar(\Met)$.
Indeed, for every $i\in I$ we have the concrete functor
 $$
 P^i \colon \cw \to \cv^i
 $$
 that  assigns to a $\Sigma$-algebra $A$ the reduct considering only the operations of $\Sigma^i$. It is clear that the reduct satisfies the equations of $\ce^i$, thus, $P^iA \in \cv^i$. 

%%%%%%%%%%%%%%%%%%%%%%%%%%%%%%%%%%%%%%%%%%%%%%%%%%%%%%%2.cast
This cone makes $(\Sigma, \ce)$-$\Met$ a product $\prod\limits_{i\in I} \cv^i$ that $\Phi^{\op}$ takes to a coproduct in $\Mnd_f(\Met)$.
 To verify this, we apply Theorem \ref{T:Beck}: let $T$ be a finitary monad and  $Q^i\colon \Met^T \to \cv^i$ ($i\in I$)  a cone in $\cv ar(\Met)$. For every algebra $\alpha \colon TA \to A$ of $\Met^T$ we obtain algebras
 $Q^i(A, \alpha)$ on the space  $A$ in $\cv^i$ ($i\in I$), which yields an algebra $F(A, \alpha)$ in $\cw$. 
 This defines a concrete functor $F = \langle Q_i\rangle_{i\in I} \colon \Met^T \to \cw$.
 This  is the unique concrete functor with $Q^i = P^i F$ ($i\in I$).

Let the monad morphism corresponding to $P^i$ be $\pi_i \colon T_{\cv_i} \to T_\cv$ (Proposition \ref{P:BW}).
 
 In $\Mnd_f (\Met)$, given a cocone $\psi_i\colon T_{\cv_i}\to S$ ($i\in I$), the corresponding concrete functors 
 $$
 G_i \colon \Mnd^S \to \Mnd^{T_{\cv_i}}\simeq \cv_i \quad (i\in I)
 $$
 (Proposition \ref{P:BW}) yield $\langle G_i\rangle\colon \Mnd^S \to \cv$. The corresponding monad morphism $\psi \colon T_\cv \to S$ is unique with $\psi_i=\psi \cdot \varphi _i$ ($i\in I$). Thus $T_\cw$ is the coproduct of the monads $T_{V_i}$.
 \vskip 1mm
(2) Equalizers: %%%%%%%%%%%%%%%%%%%%%%%%%%%%%%%%%%%%%%
apply Proposition \ref{P:e} to $\var=0$. The fact that $\Phi^{\op}$ preserves equalizers follows from Remark \ref{R:e} (1).

 \vskip 1mm
 (3) Cotensors. Given a variety $\cv$ and a metric  space $M$, we describe a variety $M\pitchfork \cv$ having the following  natural bijections, for all categories $\Met^T$ ($T$ finitary)
 $$
\xymatrix@R=.21pc{
& \Met^T  \ar[r] &M \pitchfork \cv&\\
\ar@{-}[rrr]&&&\\
&M \ar[r]& \cv ar (\Met)(\Met^T, \cv)&
}
$$
Using Theorem \ref{T:Beck} and the full embedding of Lemma \ref{T:ff}, 
 it then follows that $M\otimes \cv$ is the tensor in $\cv ar(\Met)^{\op}$ preserved by $\Phi$.

 Let $\cv =(\Sigma, \ce)$-$\Met$, then the signature $\bar \Sigma$ of $M\otimes \cv$ 
 has as $n$-ary symbols all pairs $(m,\sigma)$ where $m\in M$ and $\sigma \in \Sigma_n$:
  $$
\bar \Sigma = \big(| M|\times \Sigma_n\big)_{n\in \N}\,.
 $$
 
 %% 21a
 
 Every term $s\in T_\Sigma V$ define terms
 $$
 s^m \in T_{\bar \Sigma } V \quad (m\in M)
 $$
 %%%%%%%%%%%%%%%%%%%%%%%%%%%%%%%%%%%%%%%%%%%%%%%%%%%%%%%%%%%%%%%%%%%%%%%%%%%%%%%%%
 %%%%%%%%%%%%%%%%%%%%%%%%%%%%%%%%%%%%%%%%%%%%%%%%%%%%%%%%%%%%%%%%%%%%%%%%%%%%%%%%%
 by the following recursion on the depth $k$ of $s$ (Example \ref{E:term}): for depth $0$ put
 $$
 x^m =x \quad (x\in V)\quad \mbox{and}\quad \sigma^m =\sigma\quad \mbox{for}\quad \sigma \in \Sigma_0\,.
 $$
 Given a term $s$ of depth $k+1$, then
 $$
 s= \sigma(s_i)_{i<n} \quad \mbox{implies} \quad s^m = (m, \sigma) (s_i^m)_{i<n}\,.
 $$
 
 For every $\bar \Sigma$-algebra $A$ we denote by $A^m$ $(m\in M)$ the  $\Sigma$-algebras given by $\sigma_{A^m} = (m, \sigma)_A$ for all $\sigma\in \Sigma$. Every evaluation  $f\colon V\to |A|$ of the variables is, of course, also an evalution in $|A^m|$, we denote it by $f_m$ $(=f)$. Then $f^{\#} \colon   
T_{\bar\Sigma}X \to A$ is carried by the same map as $f^{\#}_m\colon 
T_\Sigma X \to A^m$ (for each $m\in M$).

 The variety $M\otimes \cv$ of $\bar \Sigma$-algebras is presented by the following set of equations $\ce_1\cup \ce_2$. The set $\ce_1$ consists of all equations
 $$
 s^m =_\var t^m \quad \mbox{for}\quad s=_\var t \quad \mbox{in \ }\ce  \quad \mbox{and}\quad  m\in |M|\,.
 $$
 Whereas the set $\ce_2$ consists  of all equations
 $$
 t^m (x_i)_{i<n} =_\delta t^{m'} (x_i)_{i<n}\quad \mbox{for\ }
t\in \Sigma_n  \mbox{ and }\quad d(m,m') =\delta\ \mbox{in}\ M\,.
 $$
 Then we obtain concrete functors
 $$
 U^m \colon M\otimes \cv \to \cv\quad (m\in |M|)
 $$
 taking $A$ to $A^m$. Indeed, the equations in $\ce_1$ guarantee that $A^m\in \cv$ for every algebra $A\in M\otimes \cv$. From the equations in $\ce_2$ we conclude that
 \begin{equation*}
 d(U^m, U^{m'}) \leq d(m, m') \quad \mbox{for\ } m, m'\in |M|\,.
 \tag{$\ast$}
 \end{equation*}
 This follows easily from the fact that every term $t \in T_\Sigma n$ fulfils $t_{U^m A} =t^m _A.$ (This is easy to prove by indu tion on the depth of $t$.)
The variety $M\otimes \cv$ is a tensor  in $\cv ar (\Met)^{\op}$ or, equivalently, a cotensor in $\cv ar(\Met)$. Indeed, we have bijections
$$
\xymatrix@R=.21pc{
&\cw \ar[r] &M \otimes \cv&\\
\ar@{-}[rrr]&&&\\
&M \ar[r]& \cv ar(\Met) (\cw, \cv)&
}
$$
natural in $\cw \in \cv ar(\Met)$. To a concrete functor $H\colon \cw \to M\otimes \cv$ it assigns the map $f$ with
$f(m) \colon \cw \to \cv$ defined by
$$
f(m) = U^m H \quad (m\in |M|)\,.
$$
Then $f$ is nonexpanding, due to \thetag{$\ast$}.

The inverse passage assigns to every nonexpanding map 
$$
f\colon M\to \cv ar(\Met) (\cw,  \cv)
$$
 the unique concrete functor $H\colon\cw \to M\otimes \cv$ with $f(m) = U^m H$ ($m\in |M|$). This functor takes an algebra $A$ to the $\bar\Sigma$-algebra $HA$ on the same metric space defined by
$$
(m, \sigma)_{HA} = \sigma_{f(m)}\qquad (m\in M, \sigma \in \Sigma)\,.
$$
The algebra $HA$ satisfies $\ce_1$ due to $f(m)$ taking $ A$ into $ M\otimes \cv$ for all $m\in M$. It satisfies $\ce_2$ because $f$ is nonexpanding.

We have verified that $M\otimes \cv$ is a tensor of $\cv$ in $\cv ar(\Met)^{\op}$. The proof that $\Phi(M\otimes \cv)=T_{M\otimes \cv}$ is  the tensor $M\otimes \Phi\cv$ in $\Mnd_f(\Met)$ is completely analogous to Item (1).
\end{proof}

\section {1-Basic Monads}  % 5
We prove here that varieties of  quantitative algebras bijectively correspond to 1-basic monads. We use the enriched categories $\Mnd_f (\Met)$ and $\Mnd_f(\CMet)$ (Notation \ref{N:mnd}).
%%%%%%%%%%%%%%%%%%%

\begin{defi}
A monad on $\Met$ is \emph{$1$-basic} if it is a weighted colimit of strongly finitary monads in   $\Mnd_f(\Met)$. Analogously for $1$-basic monads on $\CMet$.
 \end{defi}

Thus we have the following implications:

$\hspace {1,5 cm}$ strongly finitary  $\Rightarrow$ 1-basic $\Rightarrow$ finitary.

The converse does not hold, see Section 7 for a 1-basic monad that is not strongly finitary. A simple example of a finitary monad that is not 1-basic is given by the full subcategory $\cv$ of $\Met$ on spaces with all non-zero distances greater or equal to $1$. (The signature here is empty.) This subcategory cannot be presented by quantitative equations, thus $T_\cv$ is   
1-basic (see Theorem \ref{T:main} below).  But $T_\cv$ is finitary: it assigns to every space the quotient modulo the equivalence merging all pairs of elements of distance smaller than $1$.  This functor preserves directed colimits by Proposition A1

%%%%%%%%%%%%%%%%%%%%%%%%% tady
\begin{theorem}\label{T:main}
 A  monad on $\Met$ or $\CMet$  has the form $T_\cv$ for a variety $\cv$ of quantitative algebras iff it is 1-basic
\end{theorem}

\begin{proof} We present a proof for $\Met$. The proof of $\CMet$ is identical.

(1) Sufficiency: Let $T$ be a 1-basic monad. We have a diagram $D\colon \cd \to \Mnd_f(\Met)$ with all monads $Dd$ strongly finitary, and a  weight  $W$ with
$$
T=\colim_W D\,.
$$
By  Theorem \ref{T:adv}, every monad $Dd$ is the free-algebra monad of a variety. Since $\Phi$ is a full embedding (Theorem \ref{T:ff}), the functor $D$ factorizes through  it via an enriched functor $D'$:
$$
\xymatrix@C=2pc@R=3pc{
& \cd \ar[dl]_{D'} \ar[dr]^{D}&\\
\cv ar (\Met)^{\op}\ar[rr]_{\Phi} && \Mnd_f(\Met)
}
$$
Put  $\cv = \lim_W (D')^{op}$ in $\cv ar (\Met)$, then $T_\cv =\colim_W \Phi D'$ (Theorem \ref{T:cp}). Thus both monads  $T$ and $\Phi\cv=T_\cv$ are colimits of $D$ weighted by $W$, therefore they are isomorphic. Hence, $T$ is also a free-algebra monad of $\cv$.

(2) Necessity: For every variety $\cv$ the monad $T_\cv$ is 1-basic. Indeed:

(2a) If $\cv$ is presented by a single equation $s=_\var t$, 
where $s$ and $t$ contain together $n$ variables, then we can assume that $s, t\in T_\Sigma n$, without loss of generality. The functors $F_s$, $F_t\colon \Sigma$-$\Met \to \Gamma$-$\Met$ of Example \ref{E:BW} (2) have $T_\cv$ as their $\var$-equalizer: see Proposition \ref{P:e}.
Thus we have an $\var$-coequalizer $c$ as follows for the corresponding monad morphisms (Example \ref{E:BW} (2)).
$$
\xymatrix@1{
T_\Gamma\ \ar@<.6ex>[r]^{\widehat t} \ar@<-.6ex>[r]_{\widehat s} &\  T_\Sigma\  \ar[r]^c &\  T_\cv
}
$$
Since both $T_\Gamma$ and $T_\Sigma$ are strongly finitary (Corollary 3.5), $T_\cv$ is 1-basic.

(2b) In general, $\cv$ is presented by a set $\mathcal E =\{e_i\}_{i\in I}$ of quantitative equations. For every equation $e_i$ in $\mathcal{E
}$ the corresponding monad $\cv_i= (\Sigma, \{e_i\})$-$\Met$  contains $\cv$, and, by (2a), $T_{\cv_i}$ is a 1-basic monad. Moreover, $\cv$ is the intersection of all of these varieties. That is, we have a wide pullback of full embeddings in $\cv ar(\Met)$:
$$
\xymatrix@C=.4pc{
& V \ar[ld] \ar[d] \ar[rd] && \\
\cv_i \ar[dr] & \ar[d]& \ar[dl]&\dots \\
&\Sigma\mbox{-}\Met & &
}
$$
By Theorem \ref{T:cp} the monad $T_\cv$ is a wide pushout of the 1-basic monads $T_{\cv_i}$ and $T_\Sigma$. Thus, $T_\cv$ is 1-basic.

We can also express $T_\cv$ directly as a weighted colimit. Let the equation $e_i$ have the form
\begin{equation*}
\tag{$e_i$}
s_i=_{\var_i} t_i \quad \mbox{for} \quad s_i, t_i\in T_\Sigma n(i)\,,
\end{equation*}
and let $\Gamma_i$ be the signature of a single symbol of arity $n(i)$. The enriched category $\cd$ consists of parallel pairs with a common codomain, indexed by $I$:
$$
u_i, v_i \colon r_i \to r\quad \mbox{with}\quad d(u_i, v_i)=\infty\,.
$$
 We form the following diagram
 $$
 D\colon \cd \to \Mnd_f(\Met)\colon\ r\mapsto T_\Sigma \quad \mbox{and}\quad r_i \mapsto T_{\Gamma_i}
 $$
 where
 $$
 D u_i = \widehat s_i \quad \mbox{and}\quad Dv_i = \widehat t_i\,.
 $$
 For the following weight
 $$
 W \colon \cd^{\op} \to \Met \colon  r\mapsto \{0\}, r_i\mapsto \{\square, \Diamond\} \quad \mbox{with}\quad d(\square, \Diamond) =\var_i\,,
 $$
 we have $T_\cv=\colim\limits_W D$. This follows from Theorem \ref{T:cp}, using the wide pullback above.
\end{proof}

\begin{coro}\label{C:main}
The following enriched categories are dually equivalent:
\begin{enumerate}
\item[(1)] Varieties of quantitative algebras and concrete functors: $$\cv ar(\Met) \quad \mbox{or} \quad \cv ar (\CMet)$$.
\item[(2)] 1-basic monads on $\Met$ or $\CMet$ and monad morphisms.
\end{enumerate}
\end{coro}

Indeed, the functor $\Phi\colon \cv ar (\Met)^{\op}\to \Mnd_f(\Met)$ has the codomain restriction $\Psi $ to the full subcategory of $\Mnd_f(\Met)$ on all 1-basic monads. Since $\Phi$ is fully faithful, so is $\Psi$. By Theorem \ref{T:main}, $\Psi$ is an equivalence functor.

Analogously for $\CMet.$
%%%%%%%%%%%%%%%%%%%%%%%%%%
\begin{remark}
 We do not claim that $\Mnd_f(\Met)$ is cocomplete. But every diagram of strongly finitary monads has, for each weight, a weighted colimit. 1-basic monads are precisely the resulting weighted colimits.
 \end{remark}

%%%%%%%4/1
\section{Unary Algebras} %% sec6
We prove that for each variety $\cv$ of unary quantitative algebras
  the monad $T_\cv$ is strongly finitary.

\begin{assump} %4.0
In the present section  $\cv$ denotes a variety of signature $\Sigma$ with all arities $1$.
\end{assump}

Examples include the variety  of actions of  a quantitative monoid, and the variety of $\var$-small spaces (Example \ref{E:var}).

\begin{remark}
	
	(1) Below we work with the category 
	$$\PMet$$
	of (extended) pseudometric spaces, defined precisely as $\Met$, except that the distance $0$ is allowed for distinct elements.
	
	(2) The full  subcategory $\Met$ is reflective in $\PMet$. The reflection of a pseudometric space $X$ is the quotient map
$$
q\colon X \to X/\sim \quad \mbox{where \ \ $x\sim y$\ iff $d(x,y)=0$.}
$$
The equivalence classes in $X/\sim$ have  the distances derived from  those in $X$, that is:
 $$
 q\ \ \mbox{preserves distances.}
 $$
 
 (3) All pseudometrics on a given set $X$ form a complete lattice: we put $d\leq d'$ if $d(x,y) \leq d'(x,y)$ holds for all $x, y \in X$. We next describe binary meets in detail.
 \end{remark}
 
\begin{constr}\label{C:meet}
 The meet
 $$
 d= d' \wedge d^{\prime\prime}
 $$
 of pseudometrics $d'$ and  $d''$ on a set $X$ is constructed from their pointwise minimum
 $$
 d^0 =\min \{d', d^{\prime\prime}\}
$$
as follows:
$$
d(x,y)=  \inf \sum\limits_{i<n} d^0 (s_i, s_{i+1}) \quad (\mbox{for\ } x, y\in X)\,.
$$
The infimum   ranges over all sequences $x= s_0, s_1, \dots s_{n} =y$ in $X$. (The case $n=0$ means $x=y$, and the infimum is $0$.)

Indeed, the function $d(x,y)$ is  symmetric. It satisfies the triangle inequality because we can concatenate sequences in $X$. Thus, $d$ is a pseudometric. We have $d\leq d'$ and $d\leq d^{\prime\prime}$: use sequences with $n=1$.

Finally, for every pseudometric $\widehat d$ with  $\widehat d \leq d'$ and $\widehat d \leq d^{\prime\prime}$ we have $\widehat d \leq d^0$. This implies $\widehat d \leq d$ because for every sequence $x=s_0, \dots , s_n=y$ we have, due to the triangle inequality,
$$
\widehat d (x,y) \leq \sum_{i<n} \widehat d (s_i, s_{i+1}) \leq \sum_{i<n} d^0 (s_i, s_{i+1})\,.
$$
Hence, $\widehat d (x,y) \leq d(x,y)$.
 \end{constr}
 
 We can define quantitative algebras on pseudometric spaces
and their homomorphisms
 precisely as in Definition \ref{D:quant}. 
We denote the resulting category by $\Sigma$-$\PMet$.
Then $\Sigma$-$\Met$ is a full subcategory of $\Sigma$-$\PMet$. It is in fact a reflective subcategory:
 
 \begin{lemma}\label{L:new}
 Let $A$ be an algebra  in $\Sigma$-$\PMet$. Then the metric reflection $q\colon A\to A/\sim$ admits a unique structure of a quantitative algebra on $A/\sim$ for which $q$ is a homomorphism.
 \end{lemma}

 \begin{proof}
 Given $\sigma \in \Sigma_n$, we have the operation $\sigma_A \colon A^n \to A$. If $q$ is to be a homomorphism, we must define $\sigma_{A/\sim}\colon (A/\sim)^n \to A/\sim$ by the following rule
 $$
 \sigma_{A/\sim} \big(q (a_i)\big) = q \big(\sigma_A(a_i)\big) \quad \mbox{for all $(a_i) \in A^n$}.
$$
 This formula is independent of the  choice of representatives: suppose $q (a_i) = q(a_i')$ for $i<n$, that is, 
in $A$ we have 
$\max\limits_{i<n} d(a_i, a_i') =0$. Since $\sigma_A$ is nonexpanding, this implies   $d\big(\sigma(a_i), \sigma(a_i')\big) =0$.  That is $q\big(\sigma(a_i)\big) = q\big(\sigma(a_i')\big)$.

Since $q$ preserves  distances, the operation $\sigma_{A/\sim}$ is nonexpanding. Thus $A/\sim$ is a quantitative algebra. The uniqueness of the operations on $A/\sim$ is clear.
\end{proof}

\begin{coro}\label{C:refl}
The homomorphism $q\colon A\to A/\sim$ is a reflection  of $A$ in $\Sigma$-$\Met$: Given an  algebra $B$ in $\Sigma$-$\Met$  and a nonexpanding homomorphism $f\colon A\to B$, there is a unique nonexpanding homomorphism $\bar f \colon A/\sim \to B$ with $f=\bar f \cdot q$.
\end{coro}

Indeed, define $\bar f\big(q(a)\big) = f(a)$. This is independent of the choice of $a$, and yields the desired homomorphism $\bar f$.

 \begin{nota}
 Let $\cv$ be a variety of quantitative algebras. 
 
 (1)We write $\cv \vdash t=_\var t'$ if every algebra of $\cv$ satisfies $t=_\var t'$.
 
 (2)  For every space $X$ we define a  pseudometric $\bar d_X^{\cv}$ on the set $T_\Sigma |X|$ of terms as follows:
$$
 \bar d_X^\cv (t, t') =\inf \{\varepsilon \geq 0; \cv \vdash t=_\varepsilon t'\}\,.
 $$
 It is easy to see that this is a pseudometric: the triangle inequality follows from the fact that if an algebra satisfies 
 $t=_\var t'$ as well as $t' =_\delta t''$, then it also satisfies $t=_{\var + \delta} t''$.

(3) We put 
  $$
 d_X^\cv = d_X^\ast \wedge \bar d_X^\cv
 $$
 for the metric $d_X^\ast$ of the free algebra (Example \ref{E:term}).
  \end{nota}
 
 \begin{lemma} All operations of $\Sigma$ are nonexpanding with respect to $d_X^\cv$.
 \end{lemma}
 
 \begin{proof}
We verify that for all terms $t, t' \in T_\Sigma |X|$ we have
 $$
 d_X^\cv (t, t') \geq d_X^\cv \big(\sigma(t), \sigma(t')\big)\,.
 $$
 Denote by $d_1$ the following pseudometric on $T_\Sigma |X|$:
 $$
 d_1 (t, t') = d_X^\cv \big(\sigma(t), \sigma(t')\big)\,.
 $$
 Since $\sigma$ is nonexpanding, we have  $d_1 \leq d_X^\ast$:
 $$
 d_1 (t, t') \leq d_X^\ast \big(\sigma(t), \sigma(t')\big) \leq d_X^\ast (t, t')\,.
 $$
 Further, we have $d_1 \leq \bar d_X^\cv$ because, whenever $t=_\var t'$ is satisfied in $\cv$, then  $\sigma(t)=_\var(t')$ is also satisfied. (This follows from the fact that $f^{\#}\colon TX \to A$ preserves the operations given by $t$ and $t'$, see Notation \ref{N:term}).
  Thus,
  $$
   d_1 (t, t') \leq \bar  d_X^\cv  \big(\sigma(t), \sigma(t')\big)\leq \varepsilon\,.
 $$
 Consequently,
 $
 d_1\leq d_X^\cv\,.
 $
This means precisely that $\sigma$ is nonexpanding.
  \end{proof}

 \begin{theorem}\label{T:f}
 Let $A$ be the algebra $T_\Sigma|X|$  of terms endowed with the pseudometric $d^\cv_X$. 
 The free quantitative algebra $T_\cv X$ on  $X$ is the metric reflection $q\colon A\to A/\!\sim$  with the universal map $q \cdot \eta_X$.
 \end{theorem}
 
 \begin{proof}
 (1) The metric reflection $q \colon A \to A/\sim$ (where $t\sim t'$ means $d_X^\cv (t, t') =0$) yields an algebra in $\cv$. Indeed, let $s, s' \in T_\Sigma$ be terms such that $\cv$ satisfies $s=_\var s'$, then we verify
$$
d\big( f^{\#}(s),  f^{\#}(s') \big) \leq \var \quad \mbox{for each}\ \ f\colon V \to A/\sim \,.
$$
This is trivial for $X$ empty. Otherwise, choose a splitting of $q$ in $\Set$:
$$
i\colon |T_\cv X| \to T_\Sigma |X| \quad \mbox{with\quad} q\cdot i =\id\,.
$$
For the interpretation 
$$
g= i\cdot f \colon V\to T_\Sigma|X|
$$
we have, by Corollary \ref{C:refl}, a nonexpanding homomorphism $g^{\#}\colon T_\Sigma V \to (T_\Sigma|X|, d_X^\cv)$
with $g^{\#} \cdot \eta_v = i \cdot f$. The outward triangle of the following diagram
$$
\xymatrix@R=8mm@C=1.9cm{
& T_\Sigma V \ar[dddl]_{g^{\#}} \ar[dddr]^{f^{\#}}&\\
&V\ar[u]^>>>>>>>{\eta_V} \ar[d]_{f} &\\
& T_\cv X \ar[dl]_{i} \ar[dr]^{\id} &\\
 T_\Sigma|X| \ar[rr]_{q} && T_\cv X
 }
 $$
 commutes because it does when precomposed by $\eta_V$.
 Put $t= g^{\#}(s)$ and $t' =g^{\#}(s')$, then
 $$
 t=_\var t' \quad \mbox{holds in \ \ $\cv $}.
 $$
 Indeed, let $B$ be an algebra in $\cv$ and $h\colon |X| \to B$ an interpretation. Since $B$ satisfies $s=_\var s'$, the interpretation $k= h^{\#} \cdot g \colon V\to B$ fulfils
 $$
 d_B \big(k^{\#} (s), k^{\#}(s') \big) \leq \var\,.
 $$
 We have $k^{\#} = h^{\#} \cdot g^{\#}$ (both sides are nonexpanding homomorphisms which extend $h^{\#} \cdot g$). Thus from $t= g^{\#}(s)$ and $t'= g^{\#}(s')$ we get
 $$
 d_B \big(h^{\#}(t), h^{\#}(t')\big) \leq \var\,.
 $$ 
 From $\cv \vdash t=_\var t'$  we derive
 $$
 \bar d_X^\cv (t, t') \leq \var\,,
 $$
 thus,
 $$
 \bar d_X^\cv \big(g^{\#}(s), g^{\#}(s')\big) \leq \var\,.
 $$
 Since $q$ preserves distances and $q \cdot g^{\#} = f^{\#}$, this proves
$$
 d_X^\cv \big(f^{\#}(s), f^{\#}(s')\big) \leq \var\,,
$$
as desired.

(2) We verify the universal property of $q \cdot \eta_X \colon X \to A/\sim$. Given an algebra $B$ in $\cv$ and a nonexpanding map $f\colon  X\to B$, we present a nonexpanding homomorphism  $\bar f\colon T_\cv X\to B$ making the following square commutative:
$$
\xymatrix  @C=3pc@R=3pc{
X\ar[r]^f \ar[d]_{\eta_X} & B\\
T_\Sigma |X| \ar[r]_{q} & T_\cv X \ar[u]_{\overline f}
}
$$
For $f^{\#} \colon T_\Sigma \to A$ we define a pseudometric on 
$T_\Sigma |X|$ by
$$
d(t, t') = d_A \big( f^{\#}(t), f^{\#}(t')\big)\,.
$$
We verify that
$$
d\leq d_X^\cv\,.
$$
Since $f^{\#}$ is nonexpanding with respect to $d_X^\ast$ (Remark \ref{L:free}), we have 
$$
d\leq d_X^\ast\,.
$$
In order to prove $d\leq \bar d_X^\cv$, consider an equation
$$
\cv \vdash t=_\var t' \qquad (t, t' \in  T_\Sigma |X|).
$$
As $B\in \cv$ satisfies $t=_\var t'$, we have  $d(t, t') = d_B \big(f^{\#}(t), f^{\#}(t')\big)\leq \var$. 

Since $d\leq \bar d_X^\cv$, we see that
$$
f^{\#} \colon (T_\Sigma |X|, \bar d_X^\cv) \to B
$$
 is a nonexpanding homomorphism. By Corollary \ref{C:refl} there exists a unique homomorphism $\bar f$ making the following triangle commutative:
$$
\xymatrix  @C=2pc@R=3pc{
(T_\Sigma |X|, d_X^\cv)\ar[dr]_{f^{\#}} \ar[rr]^{q}&& T_\cv X\ar[dl]^{\bar f}\\
& B&
}
$$
This is the desired morphism:  
$$
f= f^{\#} \cdot \eta_X = \bar f\cdot (q \cdot \eta_X)\,.
$$
Since $q$ is surjective, the unicity of $f^{\#}$ follows from the universal property of $\eta_X$.
\end{proof}

\begin{remark}\label{R:f} 
%We denote by $T_\cv$ the monad of free algebras of a variety $\cv$ on $\Met$. 
(1) Given a homomorphism $h\colon A\to B$ and a surjective homomorphism $e\colon A\to A'$, then every nonexpanding  map $h'\colon A' \to B$ with $h=h'\cdot e$ is also a homomorphism. This follows from the fact that $e^n$ is surjective ($n\in \N$).

(2) The underlying  functor $T_\cv$ assigns to every metric space $X$ the quantitative algebra $T_\cv X = A/\sim$ of Theorem \ref{T:f}. To a nonexpanding map $f\colon X\to Y$ it assigns  the unique nonexpanding map $T_\cv f$ making the following square commutative
$$ \xymatrix@=3pc{
T_\Sigma |X| \ar[r]^{T_\Sigma f} \ar[d]_{q_X} & T_\Sigma |Y|\ar[d]^{q_Y}\\
T_\cv X \ar[r]_{T_\cv f} & T_\cv Y
}
$$
This square determines $T_\cv f$ uniquely: it is  a homomorphism because $q_Y\cdot T_\Sigma f$ is a homomorphism and $q_X$ is a surjective   homomorphism. Now apply the fact that $q_X$ is a reflection of $(T_\Sigma |X|, \bar d_X)$ in $\Sigma$-$\Met$ (Corollary \ref{C:refl}).

The unit $\eta_X^{T_\cv}$ of $T_\cv$ is given by $\eta_X^{T_\cv}= q_X \cdot \eta_X \colon X\to T_\cv X$, and the multiplication $\mu_X^{T_\cv}\colon T_\cv T_\cv X \to T_\cv X$ is the unique  nonexpanding homomorphism with  $\mu_X^{T_\cv}\cdot \eta_{T_\cv X}^{T_\cv} =\id$.
\end{remark}
\begin{coro}\label{C:f} For every space $X$ the map $T_\cv i_X$ is an isomorphism.
\end{coro}

\begin{lemma}\label{L:set}
Let $s$, $ s'\in T_\Sigma V$ be terms over a set  $V\ne \emptyset$,
and let $v\colon V \to W$ be an injective map. Given a quantitative algebra $A$  satisfying the equation
$$
T_\Sigma v (s) =_\var T_\Sigma v(s')\,,
$$
 then $A$ satisfies $s=_\var s'$.
\end{lemma}

\begin{proof}
Every evaluation $f\colon V\to A$ has the form $f= g\cdot v$ for some map $g\colon W\to A$. Since $T_\Sigma v \colon T_\Sigma V \to T_\Sigma W$ is a homomorphism, the following triangle commutes 
$$
\xymatrix  @C=2pc@R=3pc{
T_\Sigma V\ar[dr]_{f^{\#}} \ar[rr]^{T_\Sigma v}&& T_\Sigma W\ar[dl]^{g^{\#}}\\
& A&
}
$$
As $A$ satisfies $T_\Sigma v(s) =_\var T_\Sigma v(s')$, we obtain
$$
d\big(f^{\#}(s), f^{\#}(s')\big) = T\big( g^{\#}\big(T_\Sigma (v(s))\big), g^{\#}\big(T_\Sigma v(s')\big) \leq \var\,. 
$$
\end{proof}

%%%%%%%%%%%%%%%%%%%%%%%%%%%%%%%%%%%%%%%

\begin{theorem} 
For every variety $\cv$ of unary quantitative algebras the monad $T_\cv$ is strongly finitary.
\end{theorem}

\begin{proof}
We apply Proposition \ref{P:comp}. The monad $T_\cv$ is finitary (Lemma \ref{C:fin}), and $T_\cv i_X$ is epic (Corollary   \ref{C:f}). 
Thus, our task is to prove, for every nonexpansive map $f\colon T_\cv |X| \to Y$ satisfying
\begin{equation*}
d_Y (f\cdot T_\cv l_\var, f\cdot  T_\cv r_\var)\leq \var \quad \mbox{for every}\quad \var > 0\,,
\tag{$\ast$}
\end{equation*}
that $f$ factorizes through $T_\cv i_X$.
Throughout the proof, the metric reflection of $(T_\Sigma |X|, d_X^\cv)$ is denoted by
$$
q_X \colon (T_\Sigma |X|, d_X^\cv) \to T_\cv X\,.
$$

 We define  the following pseudometric $d$ on $T_\Sigma |X|$:
$$
d(t, t') = d_Y \big(f\cdot q_{|X|} (t), f\cdot q_{|X|} (t')\big)\,,
$$
for all terms $t$, $t'$.

(1) We first verify that
$$
d\leq\min \{ d_X^\ast, \bar  d_X^\cv\}
$$

(1a) Proof of $d\leq d_X^\ast$. If $t$, $t'$ are non-similar terms, then $d_X^\ast (t, t') =\infty$, and there is nothing to prove. Let $t$ and $t'$ be similar. The term $t$ contains a single variable $x$. 
  Then $t'$ is the term obtained from $t$ by substituting $x$ by $x'$. Put $d_X^\ast (t, t')=\var$, then we are to prove
$$
d(t, t') \leq \var\,.
$$
From the definition of $d_X^\ast$ it follows that $d_X (x, x') =\var$.
 Thus $(x, x') \in\Delta_\var X$. Let $s$ be the term 
in $T_\Sigma (\Delta_\var X)$
obtained from $t$ by substituting  $x$ by $(x, x')$. Then 
$$
t= T_\Sigma l_\var (s)\quad \mbox{and}\quad t' = T_\Sigma r_\var (s)\,.
$$
  Due to  \thetag{$\ast$} we conclude
   the desired inequality $d(t, t')\leq \var$.

(1b) Proof of $d\leq \bar d_X$. Our task is to verify that given an equation $t=_\var t'$ (for terms $t$, $t' \in T_\Sigma|X|$) holding in $\cv$, then $d(t,t') \leq \var$. Consider the algebra  $T_\cv |X|$ and the interpretation 
$$
h= q_{|X|} \cdot \eta_{|X|} \colon |X| \to T_\cv |X|\,.
$$
We know that $h^{\#} (t)$ and $h^{\#}(t')$ have distance at most $\var$. Moreover
$$
h^{\#} = q_{|X|} \colon T_\Sigma |X| \to T_\cv |X|
$$
because both sides are  nonexpanding homomorphism which extend $q_{|X|} \cdot \eta_{X}$. Thus
$$
d_{T_\cv X} \big( q_{|X|} (t), q_{|X|}(t')\big)\leq \var\,.
$$
Since $f\colon T_\cv {|X|} \to Y$ is nonexpanding, this yields the desired inequality:
$$
d(t, t') = d_Y \big(f \cdot q_{|X|} (t), f\cdot  q_{|X|}(t')\big)\leq \var\,.
 $$
 
 (2) As $d$ satisfies the triangle inequality, from \thetag{1} we get $d\leq d_X^\cv$. Hence, given $t$, $t'\in T_\Sigma |X|$ we have
 \begin{equation*}
 d_Y \big(f \cdot q_{|X|} (t), f\cdot  q_{|X|}(t')\big)\leq d_X^\cv (t, t')\,.
 \tag{$\ast$}
 \end{equation*}
 As $T_\cv i_X$ is an isomorphism (Corollary \ref{C:f}),
we can thus  define $g\colon T_\cv X\to Y$ in an element $x =T_\Sigma i_X \cdot  q_X (t)$ by
 $$ g(x) = f\cdot q_{|X|} (t)\,.
 $$
 This mapping is not only well-defined, it is nonexpanding. 
 
 In $\Met$ we have a commutative diagram as follows:
 $$
\xymatrix  @C=2pc@R=3pc{
T_\Sigma |X| \ar[r]^{q_{|X|}} \ar[d]_{T_\Sigma i_X} & T_v|X| \ar[r]^f \ar[d]_{T_v i_X} & Y\\
T_\Sigma X \ar[r]_{q_X} & T_vX \ar[ur]_g
}
$$
Indeed, the square clearly commutes and, since $T_\Sigma i_X$ is carried by identity, the outward shape commutes, too (by definition of $g$). Thus, the right-hand triangle commutes, because it does when precomposed by $q_{|X|}$. Therefore $g$ is the desired factorization of $f$.%\hfill \qed
 \end{proof}

 \section{A Counter-Example} 
 
  In the present section we prove that the free-algebra monad $T_\cv$ for the variety of two $\var$-close binary operations  is not strongly finitary.

\begin{assump} Throughout this section $\Sigma =\{\sigma_1, \sigma_2\}$ with $\sigma_1$, $\sigma_2$ binary.  For a fixed number $\var$ with $0<\var<1$ we denote by $\cv$ the variety presented by the  quantitative equation
  $$
  \sigma_1 (x,y) =_\var \sigma_2(x,y)\,.
   $$
  \end{assump}
  
 \begin{remark} 
 In universal algebra the free algebra $T_\Sigma V$ on a set $V$ of variables can be represented as follows. The elements are 
 all finite, ordered binary trees with leaves labelled in $V$, and inner nodes labelled by $\sigma_1$ or $\sigma_2$.
Here 
 trees with a label-preserving isomorphism between them are identified.
 The operation $\sigma_i$ is tree-tupling with the root labelled by $\sigma_i$.
  That is, the  variable $x\in V$ is represented by the root-only tree labelled by $x$.  
The term $\sigma_i(t_l, t_r)$ is represented by the tree below 
$$
\xymatrix  @C=4mm@R=1cm{
&\sigma_i \ar@{-}[dl] \ar@{-}[dr]&\\
t_l & & t_r
}
$$
\end{remark}

\begin{nota}\label{N:max} 
For every metric  space $X$ we define  the following metric
$$
\widehat d_X
$$
on the set $T_\Sigma|X|$ of all terms. For all variables $x$ and $y$ we use their distance in $X$
$$
\widehat d_X (x,y) = d(x,y)\,,
$$
and we put
$$
\widehat d_X (x,t) = \infty \quad \mbox{if}\quad t\notin |X|\,.
$$
All other distances $\widehat d_X (t,t')$ are defined by recursion: Represent $t$ and $t'$ as the following trees
\begin{equation*}
t=
\xymatrix  @C=4mm@R=1cm{
&\sigma_i \ar@{-}[dl] \ar@{-}[dr]&\\
t_l & & t_r
}
\quad \quad
t'=
\xymatrix  @C=4mm@R=1cm{
&\sigma_j \ar@{-}[dl] \ar@{-}[dr]&\\
t'_l & & t'_r
}\tag{*}
\end{equation*}

\noindent 
Let $m$ denote  the maximum of the distances $\widehat d_X(t_l, t_l')$ and $\widehat d_X(t_r, t_r')$. Put
$$
\widehat d_X (t, t') =\begin{cases}
m & \mbox{if}\ \ i=j\\ \var +m &\mbox{else.}
\end{cases}
$$

The \emph{depth} $\Delta (t)$ of a tree is defined by recursion: $d(x)=0$ for variables $x$, and 
$$
\Delta \big(\sigma_i(t_l, t_r)\big) = \max \big\{ \Delta (t_l), \Delta(t_r)\big\} +1\,.
$$
\end{nota}

\begin{lemma}
The free algebra $T_\cv X$  on a metric space $X$ is the space
$$
\big( T_\Sigma |X|, \widehat d_X\big)
$$
with operations  given by tree-tupling. The universal map is the inclusion morphism $X\hookrightarrow  T_\Sigma |X|$.
\end{lemma}

\begin{proof} (1) The map  $\widehat d_X$ is symmetric, and it satisfies $\widehat d_X (t, t')=0$ iff $t=t'$. This is easy to prove by induction
on the maximum of the depths of $t$ and $t'$.
 The triangle inequality
$$
\widehat d(t, t')  + \widehat d(t', t^{\prime\prime}) \geq d(t', t^{\prime\prime})
$$
also follows by induction, using  the fact that, whenever $d(t, t') \ne \infty$, then the terms $t$ and $t'$ are similar (Example \ref{E:term}).

%%%
(2) The operation $\sigma_1$ is nonexpanding:  given $\delta$ such that
$$
\widehat d_X (t_l, t_l') \leq \delta \quad \mbox{and}\quad  \widehat d_X (t_r, t_r') \leq \delta\,,
$$
we verify $\widehat d\big(\sigma_1(t_l, t_r), \sigma_1 (t_l', t_r')\big) \leq \delta$. Indeed, for the trees $t$ and $t'$  in \thetag{*}  with  $i=1=j$ we have $d(t, t') =m \leq \delta$.

Analogously for $\sigma_2$.

(3) $\eta \colon  X \to T_\cv X$ is nonexpanding: $\widehat d_X  (x, y) = d(x,y)$ for all $x, y\in X$.

(4) Let $A$ be an algebra in $\cv$. Given a nonexpanding morphism $f\colon X \to A$, there is a  unique homomorphism of the underlying  $\Sigma$-algebras $f^{\#} \colon T_\Sigma |X| \to |A|$ extending $f$. It  is our task to verify that $f^{\#}$ is nonexpanding:
$$
\widehat d_X (t, t') \geq d\big(f^{\#} (t), f^{\#} (t')\big) \quad \mbox{for}\ \ t, t'\in T_\Sigma |X|\,.
$$
This is clear if $t=x$ is a variable: either the  left-hand side is $\infty$, or $t'=y$ is also a variable, then, since $f$ is nonexpanding, we get
$$
\widehat d_X (x,y) \geq d\big(f(x), f(y)\big) = d\big( f^{\#} (x), f^{\#} (y)\big)\,.
$$
Now consider $t= \sigma_i (t_l, t_r)$ and $t'=\sigma_j(t_l', t_r')$. The proof of our inequality is  by induction on the maximum, $k$, of the depths of the trees $t$ and $t'$. The case $k=0$ has  just  been discussed.

In the induction step we use that  the operation $\sigma_i^A$ on $A$ is nonexpanding for $i=1,2$. By induction hypothesis we have 
$$
\widehat d_X  (t_l, t_l') \geq d\big( f^{\#}(t_l), f^{\#} (t_l')\big)\,;
$$
analogously for $t_r$, $t_r'$. For the  maximum $m$ (Notation \ref{N:max})  we thus get
$$
d\big( f^{\#}(t_l), f^{\#} (t_l')\big)\leq m \quad \mbox{and}\quad d\big( f^{\#}(t_r), f^{\#} (t_r')\big)\leq m\,.
$$
As the operation $\sigma_i^A$   is nonexpanding, this implies
$$
d\big( \sigma_i^A \big(f^{\#}(t_l), f^{\#} (t_r)\big), \sigma_i^A 
\big( f^{\#}(t_l'), f^{\#} (t_r')\big)\big) \leq m\,.
$$
Now  $f^{\#}(t)= \sigma_i^A \big(f^{\#}(t_l), f^{\#}(t_r)\big)$ , analogously for $f^{\#}(t')$. Thus
$$
m\geq d\big(f^{\#}(t), f^{\#}(t')\big)\,.
$$
This is the desired inequality in case $i=j$: we have  $\widehat d_X(t, t') =m$.

Let $i\ne j$. We use the following tree
$$ 
t^{\prime\prime}=
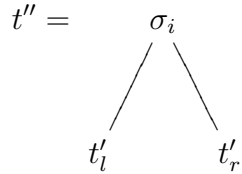
\xymatrix@C=3mm@R=12mm{
&   {\sigma_i} \ar@{-} [dl] \ar@{-}[dr]&\\
t'_l && t'_r}
$$
Since it differs from $t'$ only in the root labels, and $A$ satisfies $\sigma_1(x,y) =_{\var} \sigma_2(x,y)$, we have
$$
d\big(f^{\#}(t^{\prime\prime}), f^{\#}(t^\prime)\big)\leq \var\,.
$$
By the above case we know that
$$
m=\widehat d_X (t, t^{\prime\prime}) \geq d\big(f^{\#}(t) f^{\#}(t^{\prime\prime})\big)\,.
$$
Therefore
\begin{align*}
d(t, t') &= m+\var\\
&\geq d\big(f^{\#}(t), f^{\#}(t^{\prime\prime})\big) + d\big( f^{\#}(t^{\prime\prime}), f^{\#}(t^\prime)\big)\\
&\geq  d\big(f^{\#}(t), f^{\#}(t^{\prime})\big)\,.
\end{align*}
\end{proof}

\begin{coro}
The monad $T_\cv$ is given by $X\mapsto (T_\Sigma |X|, \widehat d_X)$. It takes a morphism $f\colon X\to Y$ to the morphism $T_\cv f$ assigning to a tree $t$ in $T_\Sigma |X|$ the tree  in $ T_\Sigma |Y|$ obtained by relabelling all the leaves from $x$ to $f(x)$.
\end{coro}

\begin{prop} \label{P:contra} The functor $T_\cv$ is not strongly finitary.
\end{prop}

\begin{proof}
%We apply Proposition \ref{P:comp}.

%(1) Let us recall that all metrics on a given set form a poset with the pointwise order. The meet
%$$
%d= d'\wedge d^{\prime\prime}
%$$
%of metrics $d'$, $d^{\prime\prime}$ is obtained from their pointwise minimum (assigning to $x$, $y$ the value $d_0=\min\{d'(x,y), d^{\prime\prime}(x,y)\}$ by the establishing the triangle inequality. More detailed: for all $x, y \in V$  define
%$$
%d(x,y) = \inf \sum_{i=1}^{n-1} d_0(s_i, s_{i+1})
%$$
%where the infimum ranges over all sequencesin $V$ as follows
%$$
%t=s_0, s_1, \dots , s_n = t'\,.
%$$
 We are going to present spaces $X$ and $Y$ and a nonexpanding map $f\colon T_\cv |X| \to Y$ satisfying \eqref{cond} of Proposition \ref{P:comp}, which does not factorize through $T_\cv i_X$. This proves our proposition. Let $X$ be the following space 
$$
X=\{a, b\}\,, \quad d(a,b)=1\,.
$$
Thus $T_\cv |X|$ is the space of all terms on $\{a,b\}$ with the metric $\widehat d_{|X|}$. The space $Y$ is the same  set  of terms with the meet of the
metrics $d_X^\ast$ (Example \ref{E:term}) and $\widehat d_{|X|}$ (Notation \ref{N:max}):
$$
Y=\big(T_\Sigma |X|, d\big)\,, \mbox{\ where\ } d= d^\ast_X \wedge \widehat d_{|X|}\,.
$$
The identity-carried map $f\colon T_\cv |X| \to Y$ is clearly nonexpanding. To verify \eqref{cond}, consider
a number $\delta >0$ and 
an arbitrary tree $u\in T_\cv (\Delta_\delta X)$. This  is a binary  tree with $k$ leaves labelled (from left to right)  by pairs $(x_i, x_i')$, $i<k$, with  $d_X(x_i, x_i')\leq 
\delta$. The tree $t=T_\cv l_\delta(u)$ is the same one, except that the $i$-th leaf label is $x_i$; analogously 
$t'= T_\cv r_\delta (u)$. The definition of $d^\ast_X$ yields
$$
d^\ast_X \big(T_\cv l_\delta (u), T_\cv r_\delta(u)\big) = d^\ast_X (t, t') =\max\limits_{i<k} d_X (x_i, x_i') \leq \delta\,.
$$
The condition \eqref{cond} states precisely this inequality, since $f$ is carried by identity.

We now prove that $f$ does not factorize through  $T_\cv i_X$. Since both $f$ and $T_\cv i_X$ are identity-carried,
 this means that $\widehat d_X (t, t') < d(t, t')$  for some trees. Indeed,  for the following trees
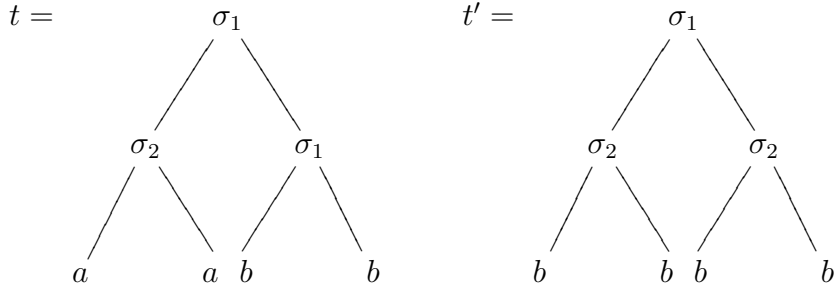
\begin{center}
$
t=
\xymatrix@C=3mm@R=12mm{
&  &  {\sigma_1} \ar@{-} [dl] \ar@{-} [dr]&&\\
& \sigma_2\ar@{-} [dl] \ar@{-} [dr] & & \sigma_1 \ar@{-} [dl] \ar@{-}[dr] &\\
a   & &a\ \ b&&b
}
$
\qquad 
$
t'=
\xymatrix@C=3mm@R=12mm{
&  &  {\sigma_1} \ar@{-} [dl] \ar@{-} [dr]&&\\
& \sigma_2\ar@{-} [dl] \ar@{-} [dr] & & \sigma_2 \ar@{-} [dl] \ar@{-}[dr] &\\
b   & &b\ \ b&&b
}
$
\end{center}
\noindent 
we verify that $\widehat d_X (t, t')=1$ and $d (t, t') =\var +1$.

The first equality follows from
$$
\widehat d_X (t_l, t_l')=1 \quad \mbox{and} \quad \widehat d_X (t_r, t_r') =\var <1\,,
$$
since we use the maximum for $\widehat d_X (t,t')$.

For the second equality observe first that
$$
\widehat d_{|X|} (t, t') =\infty = d^\ast_X (t, t')\,.
$$
(Since $d_{|X|}(a,b)=\infty$, we get $\widehat d_{|X|}(t, t')=\infty$. Since $t$ is not similar to $t'$, we have $d_X^\ast (t, t') =\infty$.)

%%%%%%%%%%%%%%%%%%%%%%%
The infimum defining $d= d_X^\ast \wedge \widehat d_{|X|}$ thus uses a sequence $t= s_0, s_1, \dots, s_n=t'$ of length $n>1$ (Construction \ref{C:meet}). One such sequence is $s_0, s_1, s_2$ where we put
$$
s_1=
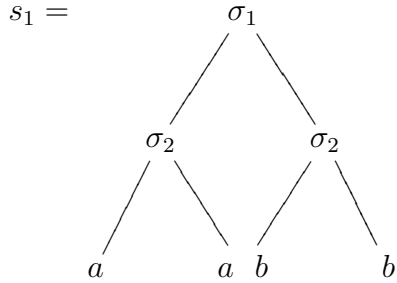
\xymatrix@C=3mm@R=12mm{
&  &  {\sigma_1} \ar@{-} [dl] \ar@{-} [dr]&&\\
& \sigma_2\ar@{-} [dl] \ar@{-} [dr] & & \sigma_2 \ar@{-} [dl] \ar@{-}[dr] &\\
a   & &a\ \ b&&b
}
$$ 
The corresponding  sum of distances is $\var +1$ due to 
 $$
  \widehat d_{|X|}(t, s_1) =\var
 $$
 and 
 $$
  d^\ast_X (s_1, t')=1\,.
 $$
 For every other sequence the sum is at least $\var +1$: let $i$ be the largest index such that $s_i$ has label $a$ at the left-most leaf. Then $d(s_i, s_{i+1})\geq  d(a,b)=1$
 because $s_{i+1}$ has  label $b$ at the left-most leaf. Since  $n\geq 2$, and $\var$ is the smallest distance  between distinct trees, $\sum\limits_{j<n} d(s_j, s_{j+1} ) \geq \var +1$.
  This proves $d (t, t') =\var +1$.
 
 In other words, $f$ does not factorize through $T_\cv i_X$, thus $T_\cv$ is not a strongly finitary functor.
 \end{proof}

 \begin{coro}
 Strongly finitary endofunctors on $\Met$ are not closed under composition.
\end{coro}

% Indeed, Theorem 4.36  presented in \cite{ADV} states  that from the compositionality of strongly finitary endfunctors it would follow that every quantitative monad is strongly finitary.
Indeed, suppose that a composite of strongly finitary endofunctors is always strongly finitary. Then the class of all finite discrete spaces  is saturated in the terminology of Bourke and Garner \cite{BG}. Their Theorem 43 then implies that $T_\cv$ is strongly finitary for every variety $\cv$.

% \section{Conclusions and Open Problems}

%We have characterized varieties of quantitative algebras as precisely the Eilenberg-Moore categories $\Met^{\T}$ where $\T$ is a semi-strongly finitary monad. This means  that $\T$  is a colimit of strongly finitary monads in the category $\Mnd_f(\Met)$ of enriched finitary monads.

%We have presented  a variety for which $\T$ is not a strongly finitary monad: the variety of algebras on two binary operations of distance $\varepsilon<1$. This also solves  in the negative  the open problem whether strongly finitary endofunctors  on $\Met$ are closed under composition.

%\begin{open}
%Which varieties $\cv$ of quantitative algebras yield strongly finitary monads $T_\cv$? Does this include all varieties presented by ordinary equations?
%\end{open}

%In the work  of Mardare et al, more general quantitative equations are considered: the $c$-basic equations, using  variables from metric spaces 
%of power smaller then $c$. A monadic characterization of varieties presented by 
%$\aleph_0$-basic equations is a highly  interesting problem.

\appendix
\section*{Appendix: Directed Colimits in $\Met$ and $\CMet$}

A  poset is directed if every finite subset has an upper bound. Colimits
of diagrams  with such domains are called directed colimits. The following proposition is a minor improvement of the proposition formulated in \cite{ADV} (where the proof is incomplete).
\vskip 3mm

\noindent
\textbf{A1\ Proposition} (Directed colimits in $\Met$).\label{P:d}
Let $D = (D_i)_{i\in I}$ be a directed diagram in $\Met$ with 
objects $(D_i, d_i)$ and 
connecting morphisms $f_{ij}\colon D_i \to D_j$ for $i\leq j$. A cocone $c_i \colon D_i \to C$ ($i\in I$) of $D$ is a colimit iff

{\rm (1)} It is collectively surjective: $C=\bigcup\limits_{i\in I} c_i [D_i]$.

{\rm (2)} The distance of elements $x, x'$ of $C$ is given by 
$$ 
d(x, x')=\inf_{j\geq i} d_j\big(f_{ij} (y), f_{ij}(y')\big)\,
$$
where $i\in I$ is an arbitrary index with elements $y, y'\in D_i$ such that $x=c_i(y)$ and $x'= c_i(y')$.

\vskip 2mm
\begin{proof}
(a) Sufficiency: suppose that conditions (1) and (2) hold. For every cocone $h_i \colon D_i \to X $ ($i\in I$) we are to find a morphism $h\colon C \to X$ with  $h_i = h\cdot c_i$ ($i\in I$). (Uniqueness
then follows from Item (1)). Define the value of $h$ in $x\in C$ as follows:
$$
h(x) = h_i(y) \quad \mbox{for any} \quad i\in I \quad \mbox{and} \quad y\in D_i \quad \mbox{with}\quad x=c_i(y)\,.
$$

\vskip 1mm

(a1) This is independent of the choice of $i$, since the poset $I$ is directed. We now verify independence of the choice of $y$. Suppose $x=c_i(y')$. Then we derive $c_i(y) = c_i(y')$ from Condition (2) applied to $x'=x$. We namely verify, for every $\var>0$, that
$$ d_i\big(c_i(y), c_i(y')\big) <\var\,.
$$
In fact, we know that the infimum of all $d_j\big(f_{ij}(y), f_{ij}(y')\big)$ is $d(x,x)=0$. Thus, some $j\geq i$ fulfils
$$
d_j\big(f_{ij}(y), f_{ij}(y')\big)<\var\,.
$$
Since  $c_i = c_j\cdot f_{ij}$, we get $d_i\big(c_i(y), c_i(y')\big) <\var$.

\vskip 1mm
(a2) The morphism $h$ is nonexpanding. Indeed, given $d(x, x')=\var$ in $C$, we prove that
$$
d\big(h(x), h(x')\big) <\delta +\var\quad \mbox{for each}\quad \delta>0\,.
$$
Since $I$ is directed, we have, by Condition (1), an index $i \in I$ and elements $y, y'\in D_i$ with $x=c_i(y)$ and $x'=c_i(y')$.
By Condition (2) there exists $j\geq i$ with
$$
d_j\big(f_{ij}(y), f_{ij}(y')\big)<\var+\delta\,,
$$
and we apply $c_i= c_j\cdot f_{ij}$, again, to get the desired inequality.

\vskip 1mm
(a3) The equality $h_i = h\cdot c_i$ is clear.

\vskip 2mm
(b) Necessity: suppose that the cocone $(c_i)$ is a colimit. We verify Conditions (1) and (2). The metric of $C$ is denoted by $d_C$.

\vskip 1mm
(b1) Denote by $m\colon C'\hookrightarrow C$ the subspace of $C$ on the union of all $c_i [D_i]$. Then we have nonexpanding maps $c_i' \colon D_i \to C'$ ($i\in I$) forming a cocone of $D$ with $c_i = m\cdot c_i'$. Let $h\colon C \to C'$ be the unique morphism with $c_i'= h\cdot c_i$ ($i\in I$). Then $m\cdot h\cdot c_i = c_i$  ($i\in I$) implies $m\cdot h=\id$, thus, $C'=C$.

\vskip 1mm
(b2) To verify Condition (2), we denote the infimum in it by $\bar d(x, x')$, and prove in (b3) that $\bar d$ is well defined and  forms a metric.
This concludes the proof: we derive that $d_C=\bar d$. Indeed, $d_C\leq \bar d$ follows from $c_i=c_j\cdot f_{ij}$:
\begin{align*}
\bar d (x, x')&= d_C\big( c_i(y), c_i(y')\big)\\
&= d_C\big(c_j\cdot f_{ij}(y), c_j\cdot  f_{ij}(y')\big)\\
&\leq d_j\big(f_{ij}(y), f_{ij}(y')\big)\,.
\end{align*}
To verify $d_C\geq \bar d$, it is sufficient to use that each $c_i$ is nonexpanding with respect to $\bar d$:
\begin {align*}
d_i(y, y') &= d_i\big(f_{ii}(y), f_{ii}(y')\big)\qquad f_{ii}=\id\\
&\geq  \inf_{j\geq i} d_j\big(f_{ij}(y), f_{ij}(y')\big)\\
&= d_C\big(c_i(y), c_i(y')\big)\,.
\end{align*}
Thus, we have a unique morphism $h\colon (C, d) \to (C, \bar d)$ with $h\cdot c_i = c_i$ ($i\in I$). By Condition (1), $h=\id$, and $d_C\geq \bar d$.

\vskip 1mm
(b3) We now prove the promised facts about $\bar d$.

(i) $\bar d(x, x')$ is independent of choice of $i$ and $y, y'\in D_i$. Since $I$ is directed, the independence of $i\in I$  is clear. We thus just need to prove that, given $z, z'\in D_i$ with $c_i(z)=x= c_i(y)$ and $c_i(z') = x'=c_i(z')$, then the two corresponding infima are equal. By symmetry, we only show that
$$
\inf_{j\geq i} d_j\big(f_{ij}(y), f_{ij}(y')\big)\leq \inf_{j\geq i} d_j\big(f_{ij}(z), f_{ij}(z')\big)\,.
$$
(The same running index $j$ can be used on both sides since $I$ is directed.) We again verify that for every $\var>0$ the inequality with the righ-hand side enlarged by $\var$ holds. From $c_i(z) = c_i(y)$,  Condition (2) yields an index $j'\geq i$ with
$$
d_{j'} \big(f_{ij'}(y), f_{ij'}(z)\big) <\var/2\,.
$$
Analogously for $z'$ and $y'$. Moreover, we can 
 assume $j'=j$, since the poset $I$ is directed. From the triangle inequality in $D_j$ we obtain
 \begin{align*}
&d_j \big(f_{ij}(y), f_{ij}(y')\big)\\
\leq&\  d_j\big(f_{ij}(y), f_{ij}(z)\big)+ d_j\big(f_{ij}(z), f_{ij}(z')\big)
+ d_j\big(f_{ij}(z'), f_{ij}(y')\big)\\
<&\ \var/2 + d_j \big(f_{ij}(z), f_{ij}(z')\big)+ \var/2
\end{align*}
as desired.

(ii) The function $\bar d$ is clearly symmetric, and it fulfils $\bar d(x,x)=0$. Since $\bar d\geq  d_C$,  for all $x\ne x'$ we have $\bar d(x, x')>0$.

It remains to prove the triangle inequality for $\bar d$. Given $x, x', x'' $ in $C$, we can  find $i\in I$ containing correponding elements $y, y', y''$ in  $D_i$. Since a directed infimum of sums (of reals) equals the sum of the corresponding infima, we get
\begin{align*}
\bar d(x, x') + \bar d(x', x'') &= \inf_{j\geq i} \big\{ d_j \big(f_{ij}(y), f_{ij}(y')+ d_j\big(f_{ij}(y'), f_{ij}(y'')\big)\big\}\\
&\geq \inf_{j\geq i} d_j\big(f_{ij}(y), f_{ij}(y')\big)\\
&= d(x, x'')\,.
\end{align*}
\end{proof}

\vskip 2mm
\noindent 
\textbf{A2\ Corollary} (Directed colimits in $\CMet$).\label{P:dir}
A cocone $c_i\colon D_i \to C$  ($i\in I$) of a directed diagram in $\CMet$ is a colimit iff

{\rm (1)} It is collectively dense: $C =\overline{ \bigcup\limits_{i\in I} c_i [D_i]}$.

{\rm (2)} The metric of $C$ is given by the formula (2) of Proposition  A1.

\begin{proof}

(a) Sufficiency. If Conditions (1) and (2) hold, then for $C' =\bigcup\limits_{i\in I} c_i [D_i]$, a dense subspace of $C$, we restrict the maps $c_i$ to nonexpanding maps $c_i' \colon D_i\to C'$. The resulting cocone (in $\Met$) is a colimit of $D$ in $\Met$ due to Proposition A1. By Remark \ref{R:Cauchy}, colimits in $\CMet$ are obtained by applying the Cauchy completion to the corresponding colimits in $\Met$. Since $C$ is complete, it is the Cauchy completion of $C'$. Thus, we conclude that $(c_i)$ is a colimit cocone in $\CMet$.

(b) Necessity. Given a colimit cone $c_i \colon D_i \to C$ ($i\in I$) in $\CMet$, and the corresponding colimit cocone
 $c_i' \colon D_i \to C'$ ($i\in I$) in $\Met$, then the canonical morphism from $C'$ to $C$ can be chosen to be an embedding of a dense subspace (Remark \ref{R:Cauchy}). Thus, $C= (C')^*.$ 

 From the fact that the cocone $(c'_i)$ satisfies (1) and (2) of Proposition A1, we conclude that  the above conditions (1) and (2)  hold for the cocone $(c_i)$.
\end{proof}

\vskip 2mm
\noindent 
\textbf{A3\ Corollary.}\label{C:finite} 
Every space $X$ in $\Met$ or $\CMet$ is the directed colimit of the diagram $D_X$ of all of its finite subspaces.

\vskip 2mm
That is, the  objects of $D_X$ are the finite subspaces, and morphisms $f\colon A\to B$ are the inclusion maps (whenever $A\subseteq B$). The colimit cocone consists of the inclusion maps into $X$. The verification  of the properties (1) and (2) above is easy.

\end{document}